\documentclass[11 pt, hidelinks, leqno]{amsart} 

\usepackage[margin=2.5cm]{geometry}
\usepackage[english]{babel}
\usepackage{hyperref}
\usepackage{amsmath}
\usepackage{amsthm}
\usepackage{amssymb}
\usepackage{ucs}
\usepackage{mathtools}
\usepackage{enumerate}
\usepackage[utf8x]{inputenc}
\usepackage[T1]{fontenc}
\usepackage{color}
\usepackage{tikz}
\usepackage{diagbox}
\usepackage{tikz-cd}
\usepackage{tkz-graph}
\usepackage{comment}
\usepackage{xcolor}

\AtBeginDocument{%
   \def\MR#1{}
}

\theoremstyle{plain}
\newtheorem{theorem}{Theorem}[section]
\newtheorem{theoremA}{Theorem}

\newtheorem*{ack}{Acknowledgments}
\newtheorem{lemma}[theorem]{Lemma}
\newtheorem{proposition}[theorem]{Proposition}
\newtheorem{corollary}[theorem]{Corollary}
\theoremstyle{definition}
\newtheorem{definition}[theorem]{Definition}
\newtheorem{example}[theorem]{Example}

\newtheorem{remark}[theorem]{Remark}

\newcommand{\R}{\mathbb{R}}

\newcommand{\C}{\mathbb{C}}
\newcommand{\N}{\mathbb{N}}

\newcommand{\A}{\mathcal{A}}

\newcommand{\Z}{\mathbb{Z}}

\newcommand{\Lcal}{\mathcal{L}}

\newcommand{\Kcal}{\mathcal{K}}
\newcommand{\Ucal}{\mathcal{U}}
\newcommand{\Ccal}{\mathcal{C}}
\newcommand{\U}{\mathcal{U}}

\newcommand{\ot}{\otimes}
\newcommand{\id}{{\rm id}}

\newcommand{\Qhom}{\mathcal{Q}\mathrm{Hom}}
\newcommand{\qhom}{\mathrm{QHom}}

\newcommand{\omegatilde}{\widetilde{\omega}}
\newcommand{\rhotilde}{\widetilde{\rho}}

\newcommand{\UNA}{\Ucal_N(A)}
\newcommand{\UNAomega}{\Ucal_{\omega,\mu}(A)}

\newcommand{\UCP}{{\rm UCP}}

\title{Operator Algebras of Universal Quantum Homomorphisms}
\author{Pierre Fima}
\address{Pierre Fima
\newline
Universit\'e Paris Cit\'e and Sorbonne Universit\'e, CNRS, IMJ-PRG, F-75013 Paris, France.}
\email{pierre.fima@imj-prg.fr}
\author{Malay Mandal}
\address{Malay Mandal
\newline
The Institute of Mathematical Sciences, X6VW+FP7, 4th Cross St, CIT Campus, Tharamani, Chennai, Tamil Nadu 600113.}
\email{malay10mandal@gmail.com}
\author{Issan Patri}
\address{Issan Patri
\newline
Theoretical Statistics and Mathematics Unit, Indian Statistical Institute, Delhi Centre, 7 S. J. S. Sansanwal Marg, New
Delhi 110016, India}
\email{issanp@isid.ac.in}
\date{}

\begin{document}

\maketitle

\begin{abstract}
Given two unital C*-algebras $A$ and $B$, we study, when it exists, the universal unital C*-algebra $\Ucal(A,B)$ generated by the coefficients of a unital $*$-homomorphism $\rho\,:\, A\rightarrow B\ot\Ucal(A,B)$. When $B$ is finite dimensional, it is well known that $\Ucal(A,B)$ exists and we study in this case properties LP, RFD, primitiveness and the UCT as well as $K$-theory. We also construct a reduced version of $\Ucal(A,B)$ for which we study exactness, nuclearity, simplicity, absence of non-trivial projection and $K$-theory. Then, we consider the von Neumann algebra generated by the reduced version and study factoriality, amenability, fullness, primeness, absence of Cartan, Connes' invariants, Haagerup property and Connes' embeddability. Next, we consider the case when $B$ is infinite dimensional: we show that for any non-trivial separable unital C*-algebra $A$, $\Ucal(A,B)$ exists if and only if $B$ is finite dimensional. Nevertheless, we show that there exists a unique unital locally C*-algebra generated by the coefficients of a unital continuous $*$-homomorphism $\rho\,:\, A\rightarrow B\ot\Ucal(A,B)$. Finally, we study a natural quantum semigroup structure on $\Ucal(A,A)$.
\end{abstract}

\section{Introduction}

\noindent In this paper, we consider quantum homomorphisms between \textit{unital} C*-algebras : given two unital C*-algebras $A$ and $B$, a quantum homomorphism from $A$ to $B$ is a unital $*$-homomorphism $A\rightarrow B\ot C$, where $C$ is some unital C*-algebra (and the tensor product is the minimal one). Note that our terminology of quantum homomorphism differs from \cite{So08} where such a quantum homomorphism is called a quantum family of maps from $B$ to $A$.

\vspace{0.2cm}

\noindent A \textit{universal quantum homomorphism from $A$ to $B$} is, if it exists, a unital $*$-homomorphism $\rho\,:\,A\rightarrow B\ot \U$ such that for any unital $*$-homomorphism $\pi\,:\,A\rightarrow B\ot C$ there exists a unique unital $*$-homomorphism $\widetilde{\pi}\,:\,\U\rightarrow C$ such that $(\id\ot\widetilde{\pi})\rho=\pi$. When it exists, such a quantum homomorphism is unique up to a canonical isomorphism. Let us call a pair of unital C*-algebras $(A,B)$ \textit{matched} if the universal quantum homomorphism from $A$ to $B$ exists and let us denote its associated unital C*-algebra by $\U(A,B)$. It is a well known observation that, for any $A$, the pair $(A,B)$ is matched whenever $B$ is finite dimensional. This observation actually goes back to Brown in 1981 \cite{Br81} in the cases $A\in\{\C^2,C^*(\Z)\}$ and $B=M_N(\C)$. For general $A$ and finite dimensional $B$, this observation follows from the work of Phillips in 1988 \cite{Ph88} who also realized that the pair $(\C^2,C^*(\Z))$ is not matched.

\vspace{0.2cm}

\noindent Some properties of the C*-algebra $\U(A,M_N(\C))$ are known : McClanahan showed that the C*-algebra  $\U(A,M_N(\C))$ has no non-trivial projections if and only if $A$ has no non-trivial projections and he also gave a computation of the K-theoretic groups of $\U(A,M_N(\C))$ for $A\in\{\C^2,C^*(\Z)\}$ \cite{McC92, McC95}. Harris showed that $\U(C^*(\Z),M_N(\C))$ is RFD \cite{Ha19}. It has also been shown by the second author that $\U(C^*(\Z),M_N(\C))$ is primitive and has the lifting property \cite{Ma23}.

\vspace{0.2cm}

\noindent In the second section of this paper, we review and generalize this results. For example we observe that for any unital C*-algebra $A$ and any finite dimensional C*-algebra $B$ one has (see Proposition \ref{PropLifting}) :
\begin{itemize}
    \item $\U(A,B)$ has LP (resp. LLP) if and only if $A$ has LP (resp. LLP).
    \item $\U(A,B)$ is RFD if and only if $A$ is RFD.
    \item If $A$ is RFD and $({\rm dim}(A)-1)({\rm dim}(B)-1)\geq 2$ then $\U(A,B)$ is primitive.
    \item If $A$ satisfies the UCT then $\Ucal(A,B)$ satisfies the UCT.
\end{itemize}
We also provide a computation of the $K$-theoretic groups of $\U(A,M_N(\C))$.

\vspace{0.2cm}

\noindent In \cite{McC92, McC95}, McClanahan introduced a natural state $\omega$ on $\U(C^*(\Gamma),M_N(\C))$, where $\Gamma$ is a discrete countable group and study the C*-algebra $\U_\omega(C^*(\Gamma),M_N(\C))$ obtained by GNS construction. He showed that if $\tau_*(K_0(C^*_r(\Gamma)))\subseteq\frac{1}{N}\Z$ then $\U_\omega(C^*(\Gamma),M_N(\C))$ is simple without non-trivial projection ($N>1$ and $\vert\Gamma\vert>1$), where $\tau_*$ is the group homomorphism induced by the canonical trace on $C^*_r(\Gamma)$.

\vspace{0.2cm}

\noindent In the third section of this paper, we consider a pair $(A,B)$ of unital C*-algebras and we assume that $B$ is finite dimensional. Given two GNS-faithful states $\omega\in A^*$ and $\mu\in B^*$, we introduce a natural state $\widetilde{\omega}\in\U(A,B)^*$ and consider the C*-algebra $\U_{\omega,\mu}(A,B)$ obtained by GNS construction of $\widetilde{\omega}$. Our construction generalizes McClanahan's construction.

\vspace{0.2cm}

\noindent Write $B=\bigoplus_{\kappa=1}^KM_{N_\kappa}(\C)$ and $N:=\sum_{\kappa=1}^KN_\kappa$. We show the following.

\begin{theoremA}
The following holds.
\begin{enumerate}
    \item $\Ucal_{\omega,\mu}(A,B)$ is exact if and only if $A$ is exact.
    \item If both $\omega$ and $\mu$ are pure then $\U_{\omega,\mu}(A,B)$ is nuclear if and only if $A$ is nuclear.
    \item If $\omega$ and $\mu$ are faithful traces, $\omega$ is diffuse and $N\geq 2$ then $\Ucal_{\omega,\mu}(A,B)$ is simple with unique trace and has stable rank $1$.
\end{enumerate}
\end{theoremA}

\noindent When $B=M_N(\C)$ we obtain the following.

\begin{theoremA} The following holds.
\begin{enumerate}
\item The C*-algebras $\U_{\omega,\mu}(A,M_N(\C)))$ and $\U(A,M_N(\C)))$ are $KK$-equivalent.

    \item If $\omega$ is a faithful traces on $A$, $N\geq 2$ and $\omega_*(K_0(A))\subset\frac{1}{N}\Z$, where $\omega_*\,:\,K_0(A)\rightarrow\R$ is the group homomorphism induced by $\omega$, then $\U_{\omega,{\rm tr}}(A,M_N(\C)))$ has no non-trivial projection, where ${\rm tr}$ is the unique tracial state on $M_N(\C)$.
\end{enumerate}
\end{theoremA}

\noindent We then study the von Neumann algebra $\U''_{\omega,\mu}(A,B)$ generated by $\U(A,B)$ in the GNS construction of $\widetilde{\omega}$. We will denote by $A''$ the von Neumann algebra generated by $A$ in the GNS construction of $\omega$.

\vspace{0.2cm}

\noindent We assume that $\omega$ and $\mu$ are faithful. As before be write $B=\bigoplus_{\kappa=1}^KM_{N_\kappa}(\C)$ and $N=\sum_{k=1}^KN_\kappa$. Then $\mu=\sum_{k=1}^K{\rm Tr}(Q_\kappa\cdot)$, where $Q_\kappa\in M_{N_\kappa}(\C)$ is strictly positive. Let ${\rm Sd}(\mu)\subset\R^*_+$ be the closed multiplicative subgroup of $\R^*_+$ generated by $\{\frac{r}{s}\,:\,r,s\in{\rm Sp}(Q_\kappa),1\leq\kappa\leq K\}$, where ${\rm Sp}(Q_\kappa)$ denotes the spectrum of $Q_\kappa$. We also write, for $A\subset\R^*_+$, $A^\perp:=\{t\in\R\,:\,a^{it}=1\text{ for all }a\in A\}$. Finally, for $A\subset\R^*_+$, we use the notation $\tau(A)$ for the smallest topology on $\R$ making the maps $(t\mapsto a^{it})$ continuous for all $a\in A$ and $\tau(\mu):=\tau({\rm Sd}(\mu))$. We prove the following.

\begin{theoremA}
The following holds.
\begin{enumerate}
\item If $\omega$ is diffuse and $N\geq 2$ then $\U''_{\omega,\mu}(A,B)$ is a non-amenable, full and prime factor of type ${\rm II}_1$ or of type ${\rm III}_\lambda$ with $\lambda\neq 0$ and $$T(\U''_{\omega,\mu}(A,B))=\{t\in\R\,:\,\sigma_t^{\omega''}=\id\}\cap{\rm Sd(\mu)}^\perp.$$
Its Connes' $\tau$-invariant is the smallest topology on $\R$ containing $\tau(\mu)$ and making continuous the map $(t\mapsto\sigma_t^{\omega''})\,:\,\R\rightarrow{\rm Aut}(A'')$.
\item If ${\rm dim}(A)\geq 2$ and either $K\neq 2$ and $N\geq 2$ of $K=2$ and ${\rm dim}(A)+N\geq 5$ then $\U''_{\omega,\mu}(A,B)$ has no Cartan subalgebra.
    \item $A''$ has the Haagerup property if and only if $\U''_{\omega,\mu}(A,B)$ has the Haagerup property.
    \item If $\omega$ and $\mu$ are both traces then $A''$ is Connes embeddable if and only if $\U''_{\omega,\mu}(A,B)$ is Connes embeddable.
\end{enumerate}
\end{theoremA}

\noindent In the fourth section, we study in full generality the notion of universal quantum homomorphism. We first show the following.

\begin{theoremA}
Let $A$ and $B$ be unital C*-algebras. Suppose that $A$ is separable with ${\rm dim}(A)\geq 2$. Then, the pair $(A,B)$ is matched if and only if $B$ is finite dimensional.
\end{theoremA}

\noindent Hence, the category of unital C*-algebras is too small to consider a universal quantum homomorphism from $A$ to $B$ with infinite dimensional $B$. Actually, the natural category to consider is the category of unital locally C*-algebras. Recall that a unital locally C*-algebra is a complete Hausdorff unital $*$-algebra whose topology is defined by a family of C*-seminorms. We show that the universal quantum homomorphism always exists in this setting.

\begin{theoremA}\label{ThmUAB}
Let $A$ and $B$ be two unital locally C*-algebras. There exists a unique (up to a canonical continuous $*$-isomorphism) unital locally separable locally C*-algebra $\U(A,B)$ with a continuous unital $*$-homomorphism $\rho\,:\, A\rightarrow B\ot \U(A,B)$ such that for any unital locally separable locally C*-algebra $\U$, any continuous unital $*$-homomorphism $\pi\,:\,A\rightarrow B\ot\U$ there exists a unique continuous unital $*$-homomorphism $\widetilde{\pi}\,:\,\U(A,B)\rightarrow\U$ such that $(\id\ot\widetilde{\pi})\rho=\pi$.
\end{theoremA}

\noindent Let $\Ccal$ be the category of unital locally separable locally C*-algebras with morhisms being the continuous unital $*$-homomorphisms. Using Theorem \ref{ThmUAB} we prove in Theorem \ref{ThmAdjoint} that the tensor product functor $\Ccal\rightarrow\Ccal$, $C\mapsto B\ot C$ admits a left adjoint, for all $B\in\Ccal$, given by $A\mapsto \Ucal(A,B)$.

\vspace{0.2cm}

\noindent We also study the natural structure of quantum semigroup on the locally C*-algebra $\Ucal(A,A)$. When $A$ is finite dimensional, we show that any invertible state on $\Ucal(A,A)$ (for the convolution product) is a character and the quantum semigroup $\Ucal(A,A)$ is actually a compact quantum group if and only if $A=\C$.

\begin{ack}
The authors gratefully acknowledge funding from ReLaX, CNRS IRL 2000 and support and hospitality of ISI, Delhi, IMJ-PRG and IIT Madras. M.M. thanks the Chennai Mathematical Institute for partial support. M.M. is also supported by The Institute of Mathematical Sciences Postdoctoral fellowship. I.P. is supported by SERB (SRG/2022/001717).
\end{ack}

\subsection{Notations} All C*-algebras and Hilbert spaces throughout this paper are assumed to be separable. The same symbol $\ot$ is used for the minimal tensor product of (locally) C*-algebras as well as the tensor product of Hilbert spaces. We denote by ${\rm Sp}(A)$ the spectrum of a unital C*-algebra $A$ and by ${\rm Sp}_\A(a)$ the spectrum of an element $a$ in a unital algebra $\A$.

\section{Universal quantum homomorphisms : Matched Pairs}

\noindent In this section we introduce and study universal quantum homomorphisms, for the specific case of matched pairs of C*-algebras. The more general case of universal quantum homomorphisms associated to arbitrary pairs of C*-algebras will be studied in Section \ref{Section:general case}.

\begin{definition}
Let $A$, $B$ be unital C*-algebras.
\begin{itemize}
\item A \textit{quantum homomorphism} from $A$ to $B$ is a unital $*$-homomorphism $\pi\,:\,A\rightarrow B\ot C$, where $C$ is some unital C*-algebra, called the C*-algebra of $\pi$.
\item A pair of C*-algebras $(A,B)$ is said to be a \textit{matched pair} if there exists a quantum homomorphism from $A$ to $B$, $\rho\,:\,A\rightarrow B\ot \U$, which has the property that for every quantum homomorphism $\pi\,:\,A\rightarrow B\ot C$, there exists a unique unital $*$-homomorphism $\widetilde{\pi}\,:\,\U\rightarrow C$ such that $(\id\ot\widetilde{\pi})\rho=\pi$. Such a quantum homomorphism is said be a \textit{universal quantum homomorphism} associated to the matched pair $(A,B)$.
\end{itemize}
\end{definition}

\begin{proposition}
Let $(A,B)$ be a matched pair and $\rho\,:\,A\rightarrow B\ot \U$ a universal quantum homomorphism. The following holds.
\begin{enumerate}
\item $\rho$ is faithful and the unital $*$-algebra generated by $\{(\omega\ot\id)(\rho(a))\,:\,\omega\in B^*,\,a\in A\}$ is dense in $\U$.
\item Any universal quantum homomorphism $\pi\,:\,A\rightarrow B\ot C$ is isomorphic to $\rho$ in the sense that there exists an isomorphism $\eta\,:\,\U\rightarrow C$ such that $(\id\ot\eta)\rho=\pi$.
\end{enumerate}
\end{proposition}

\begin{proof}
$(1)$. Consider the unital $*$-homomorphism $A\rightarrow B\ot A$, $a\mapsto 1_B\ot a$. By the universal property, there exists a unital  $*$-homomorphism $\pi\,:\,\U\rightarrow A$ such that $(\id\ot\pi)\rho(a)=1_B\ot a$ for all $a\in A$. This shows that $\rho$ is faithful. 

Consider now $C\subseteq \U$, defined as the unital C*-algebra generated by $\{(\omega\ot\id)(\rho(a))\,:\,\omega\in B^*,\,a\in A\}$ and define $\pi\,:\,A\rightarrow B\ot C$ so that $(\id\ot\iota)\circ\pi=\rho$, where $\iota$ is the inclusion from $C$ to $\U$. By the universal property, there exists a unital $*$-homomorphism $\psi\,:\,\U\rightarrow C$ such that $(\id\ot\psi)\rho=\pi$. Hence the unital $*$ homomorphism $\iota\circ\psi\,:\,\U\rightarrow \U$ satisfies $(\id\ot\iota\circ\psi)\rho=\rho$. By uniqueness, we deduce that $\iota\circ\psi=\id_\U$. It follows that $\iota$ is surjective.

\vspace{0.2cm}

\noindent$(2)$. The existence of the homomorphism $\eta$ comes from the universal property of $\rho$ and the existence of its inverse comes from the universal property of $\pi$ and uniqueness.
\end{proof}

\noindent When the pair $(A,B)$ is matched, the unique universal quantum homomorphism from $A$ to $B$ will be denoted by $\rho_{A,B}$ and its C*-algebra will be denoted by $\U(A,B)$.

\begin{example}
The following is a brief list of examples of matched pairs and their corresponding C*-algebras.
\begin{enumerate}
\item It is obvious that for any unital C*-algebra $B$, $\U(\C,B)=\C$ and $\rho_{\C,B}(z)=z1_B$, $z\in\C$.
\item If $N\in\N^*$ and $A$ is any unital C*-algebra then $\U(A,\C^N)=A^{*N}$ and:
$$\rho_{A,\C^N}\,:\,A\rightarrow\C^N\ot A^{*N},\quad\rho_{A,\C^N}(a):=\sum_{i=1}^Ne_i\ot\nu_i(a),$$
where $\nu_i\,:\,A\rightarrow A^{*N}$ be the i$^{th}$-copy of $A$ in the free product and $(e_i)_i$ is the canonical orthonormal basis (onb) of $\C^N$.
\item $G_N^{nc}:=\U(\C^2,M_N(\C))$ is the non-commutative grassmanian \cite{Br81} i.e. the universal unital C*-algebra $G_N^{nc}$ generated by $u_{ij}$, $1\leq i,j\leq N$ such that $u=(u_{ij})\in M_N(\C)\ot G_N^{nc}$ is an orthogonal projection and $\rho_{\C^2,M_N(\C)}$ is the unique unital $*$-homomorphism from $\C^2$ to $M_N(\C)\ot G_N^{nc}$ mapping $e_1$ to $u$, where $(e_1,e_2)$ is the canonical onb of $\C^2$.
\item $U^{nc}_N:=\U(C^*(\Z),M_N(\C))$ is the non-commutative unitary C*-algebra \cite{Br81} i.e. the universal unital C*-algebra generated by elements $u_{ij}$, $1\leq i,j\leq N$, such that $u:=\sum_{i,j}e_{ij}\ot u_{ij}\in M_N(\C)\ot U^{nc}_N$ is unitary and $\rho_{\C^*(\Z),M_N(\C)}$ is the unique unital $*$-homomorphism from $\C^*(\Z)$ to $M_N(\C)\ot U_N^{nc}$ mapping $g$ to $u$, where $\Z=\langle g\rangle$.
\end{enumerate}
\end{example}

\noindent In fact, it is follows from the work of Phillips \cite{Ph88} (see also \cite{So08}) that the pair $(A,B)$ is matched for any unital C*-algebra $A$ and any finite dimensional C*-algebra algebra $B$. We present below a more direct and explicit approach to show this result. The next lemma allows us to reduce to the case $B=M_N(\C)$.

\begin{lemma}\label{RmkFree}
The following hold
\begin{enumerate}
    \item Suppose, for $k=1,2$, $(A,B_k)$ is matched with universal quantum homomorphism $\rho_k\,:\,A\rightarrow B_k\ot \U_k$. View $\U_k\subset \U:=\U_1*\U_2$ and define $B=B_1\oplus B_2$ with canonical projections $\chi_k\,:\,B\rightarrow B_k$. Then, the unique unital $*$-homomorphism $\rho\,:\,A\rightarrow B\ot\U$ such that $(\chi_k\ot\id)\rho=\rho_k$, $k=1,2$, is the universal quantum homomorphism from $A$ to $B$.
    \item Suppose, for $k=1,2$, $(A_k,B)$ is matched with universal quantum homomorphism $\rho_k\,:\,A_k\rightarrow B\ot \U_k$. View $\U_k\subset \U:=\U_1*\U_2$. Then, the unique unital $*$-homomorphism $\rho\,:\,A_1*A_2\rightarrow B\ot\U$ such that $\rho\vert_{A_k}=\rho_k$, $k=1,2$, is the universal quantum homomorphism from $A_1*A_2$ to $B$.
\end{enumerate}
\end{lemma}

\begin{proof}
Straightforward.
\end{proof}

\begin{proposition}\label{PropMNcase}
Let $A$ be a unital C*-algebra and $B=\bigoplus_{\kappa=1}^KM_{N_\kappa}(\C)$ a finite dimensional C*-algebra. Let $(e^\kappa_{ij})$ be the matrix units in $M_{N_\kappa}(\C)$. The pair $(A,B)$ is matched with the universal quantum homomorphism given by:

$$\rho_{A,B}\,:\,A\rightarrow B\ot\left(\underset{\kappa=1}{\overset{K}{*}}M_{N_\kappa}(\C)'\cap(A* M_{N_\kappa}(\C))\right),\,\,a\mapsto\sum_{\kappa=1}^{K}\sum_{i,j=1}^{N_\kappa}e^\kappa_{ij}\ot\mu_\kappa\left(\sum_{r=1}^{N_\kappa}e^\kappa_{ri}ae^\kappa_{jr}\right),$$
where $\mu_s\,:\,M_{N_s}(\C)'\cap(A* M_{N_s}(\C))\rightarrow \underset{\kappa=1}{\overset{K}{*}}M_{N_\kappa}(\C)'\cap(A* M_{N_\kappa}(\C))$ is the canonical inclusion.

\end{proposition}

\begin{proof}
By Assertion $(1)$ of Lemma \ref{RmkFree}, it suffices to prove the statement for $B=M_N(\C)$. For this case we follow the proof of McClanahan (when $A=C^*(\Z)$ in \cite{McC92}). Define, for $a\in A$ and $1\leq i,j\leq N$,  $\rho_{ij}(a):=\sum_{r=1}^Ne_{ri}ae_{jr}\in A* M_N(\C)$. Note that $e_{st}\rho_{ij}(a)=e_{si}ae_{jt}=\rho_{ij}(a)e_{st}$ for all $1\leq s,t\leq N$. Hence, $\rho_{ij}(a)\in \U:=M_N(\C)'\cap (A*M_N(\C))$. It is now easy to check that $\rho\,:\,A\rightarrow M_N(\C)\ot \U$, $a\mapsto\sum_{i,j}e_{ij}\ot\rho_{ij}(a)$ is a unital $*$-homomorphism. Let us now check the universal property. Let $\pi\,:\,A\rightarrow M_N(\C)\ot C$ be a quantum homomorphism. By universal property of the free product there is a unique unital $*$-homomorphism $\psi\,:\,A*M_N(\C)\rightarrow M_N(\C)\ot C$ such that $\psi(a)=\pi(a)$ for all $a\in A$ and $\psi(b)=b\ot 1$ for all $b\in M_N(\C)$. Note that $\psi(\U)\subseteq\psi(M_N(\C))'\cap(M_N(\C)\ot C)=C$. Then $\widetilde{\pi}:=\psi\vert_\U\,:\,\U\rightarrow C$ is a unital $*$-homomorphism satisfying, for all $a\in A$,
\begin{eqnarray*}
(\id\ot \widetilde{\pi})\rho(a)&=&\sum_{i,j}e_{ij}\ot\psi(\rho_{ij}(a))=\sum_{i,j,r}(e_{ij}e_{ri}\ot 1)\pi(a)(e_{jr}\ot 1)\\
&=&\sum_{i,j}(e_{ii}\ot 1)\pi(a)(e_{jj}\ot 1)=\pi(a)
\end{eqnarray*}
To show the uniqueness of $\widetilde{\pi}$, it suffices to show that the C*-algebra $\mathcal{C}\subseteq\U$ generated by $\{\rho_{ij}(a)\,:\,1\leq i,j\leq N, a\in A\}$ is actually equal to $\U$. Note that the map $\varphi\,:\,M_N(\C)\ot \mathcal{C}\rightarrow A*M_N(\C)$, $b\ot x\mapsto bx$ is a unital $*$-homomorphism, by the universal property of tensor products. We also note that $\rho(A)\subseteq M_N(\C)\ot \mathcal{C}$. Applying the first part of the proof to the map $\pi:=\rho\,:\,A\rightarrow M_N(\C)\ot \mathcal{C}$ we get the morphism $\psi\,:\,A*M_N(\C)\rightarrow M_N(\C)\ot \mathcal{C}$ which obviously satisfies $\varphi\circ\psi=\id$ and $\psi\circ\varphi=\id$. Since $\psi(\U)\subseteq \mathcal{C}$ we have $\varphi\circ\psi(\U)=\U\subseteq \varphi(\mathcal{C})=\mathcal{C}$.
\end{proof}

\begin{remark}\label{RmkIso} The proof above implies that the map $\rho_N: A* M_N(\C)\rightarrow M_N(\C)\otimes\U(A,M_N(\C))$, sending $a\in A$ to $\rho_{A,M_N(\C)}(a)$ and $b\in M_N(\C)$ to $b\ot 1$, is an isomorphism.\end{remark}

\noindent In fact, universal quantum homomorphisms for tuples $(A,B)$, where $B$ is a finite dimensional C*-algebra, admit a generalisation to the case of unital completely positive (ucp) maps, in the following sense

\begin{corollary}\label{CorUnivUCP}
Let $A,B$ be unital C*-algebras, with $B$ finite dimensional. Then, for all ucp map $\varphi\,:\,A\rightarrow B\ot C$, there exists a ucp map $\widetilde{\varphi}\,:\,\U(A,B)\rightarrow C$ such that $(\id\ot\widetilde{\varphi})\rho_{A,B}=\varphi$.
\end{corollary}

\begin{proof}
For ease of notation, let $\rho:=\rho_{A,B}$ and $\U:=\U(A,B)$. We first assume that $B=M_N(\C)$. Let $\psi\,:\,A*M_N(\C)\rightarrow M_N(\C)\ot \U$ be the isomorphism from Remark \ref{RmkIso}. Let $\varphi\,:\,A\rightarrow M_N(\C)\ot C$ be a ucp map and define $\iota : M_N(\C)\rightarrow M_N(\C)\ot C$, mapping $b\mapsto b\ot 1$. Fix a state $\omega\in M_N(\C)^*$ and define the ucp map $\widetilde{\varphi}\,:\,\U\rightarrow C$ by $\widetilde{\varphi}(x):=(\omega\ot\id)(\varphi*\iota)\psi^{-1}(1\ot x)$. Writing $\rho(a)=\sum_{ij}e_{ij}\ot\rho_{ij}(a)$ and $\varphi(a)=\sum_{ij}e_{ij}\ot\varphi_{ij}(a)$, we have $\psi^{-1}(1\ot\rho_{ij}(a))=\sum_re_{ri}ae_{jr}$. Since $M_N(\C)$ is in the multiplicative domain of $(\varphi*\iota)$ we find $\widetilde{\varphi}(\rho_{ij}(a))=(\omega\ot\id)\sum_r(e_{ri}\ot 1)\varphi(a)(e_{jr}\ot 1)=\varphi_{ij}(a)$ so $(\id\ot\widetilde{\varphi})\rho=\varphi$. The general finite dimensional case now follows using Lemma \ref{RmkFree} and induction.
\end{proof}

\begin{remark}
Notably, there is no uniqueness in general for the ucp map $\widetilde{\varphi}$, unlike the case of homomorphisms. Let us consider the case $C=\C$,  ${\dim}(A)\geq 2$, and $B=M_N(\C)$ with $N\geq 2$. We can, in this case, construct two different states $\omega_1,\omega_2\in \U(A,M_N(\C))^*$ such that $(\id\ot\omega_1)\rho=(\id\ot\omega_2)\rho$, where $\rho:=\rho_{A,M_N(\C)}$. Indeed, let $\omega_1'\in A^*$ be any GNS-faithful state and consider the ucp map $\varphi\,:\,A\rightarrow M_N(\C)$ defined by $\varphi(a):=\omega_1'(a)1$. Consider the two unital $*$-homomorphisms $A\rightarrow M_N(\C)\ot A$, $a\mapsto 1\ot a$ and $A\rightarrow M_N(\C)\ot A^{*_rN}$, $a\mapsto \sum_{i=1}^Ne_{ii}\ot\nu_i(a)$, where $*_r$ denotes the reduced free product with respect to the GNS-faithful state $\omega_1'$ and $\nu_i\,:\, A\rightarrow A^{*_rN}$ is the $i^{th}$ free copy of $A$. By the universal property, we get two unique morphisms $\pi_1\,:\,\U(A,M_N(\C))\rightarrow A$ and $\pi_2\,:\,\U(A,M_N(\C))\rightarrow A^{*_rN}$ such that $(\id\ot\pi_1)\rho(a)=1\ot a$ and $(\id\ot\pi_2)\rho(a)=\sum_{i=1}^Ne_{ii}\ot\nu_i(a)$ for all $a\in A$. It is easily seen that $\pi_1$ and $\pi_2$ are surjective. Let $\omega_2':=*_i\omega_1'$ be the free product state on $A^{*_rN}$, which is GNS-faithful. Observe that $(\id\ot\omega_1'\pi_1)\rho=\varphi=(\id\otimes\omega_2'\pi_2)\rho$. However, we have $\omega_1:=\omega_1'\pi_1\neq\omega_2'\pi_2=:\omega_2$, which can be shown as follows- as ${\rm dim}(A)\geq 2$ and $\omega_1'$ is GNS-faithful, $\omega_1'$ is not a character. Hence, there exists $a\in A$ such that $\omega_1'(a^*a)\neq\vert\omega_1'(a)\vert^2$. In particular, $a\notin\C 1$ so that $b:=a-\omega_1'(a)1\neq 0$. Note that $\omega_1'(b)=0$ and $\omega_1'(b^*b)=\omega_1'(a^*a)-\vert\omega_1'(a)\vert^2\neq 0$. Consider the element $y:=\rho_{11}(b^*)\rho_{22}(b)\in\U(A,M_N(\C))$, which is well defined as $N\geq 2$. By uniqueness, we have that for all $x\in A$ and $1\leq i,j\leq N$, $\pi_1(\rho_{ij}(x))=\delta_{ij}x$ and $\pi_2(\rho_{ij}(x))=\delta_{ij}\nu_i(x)$, so we deduce that:
$$\omega_1'\pi_1(y)=\omega_1'(b^*b)\neq 0\quad\text{and}\quad\omega_2'\pi_2(y)=\omega_2'(\nu_1(b^*)\nu_2(b))=0.$$
\end{remark}

\noindent Next, we explore some C*-algebraic properties of $\U(A,B)$. Let us recall the following definitions.

\begin{definition} Let $A,B$ be two $C^{\ast}$-algebras. Let $J$ be a two-sided closed ideal in $B$. A completely contractively positive (c.c.p.) map $\phi:A\rightarrow B/J$ is said to be liftable if there exists a c.c.p. map $\psi:A\rightarrow B$ such that $\pi\circ\psi=\phi$, where $\pi:B\rightarrow B/J$ is quotient map. A c.c.p. map $\phi:A\rightarrow B/J$ is said to be locally liftable if for every finite-dimensional operator systems $E$ of $A$, there exists a c.c.p. map $\psi:A\rightarrow B$ such that $\pi\circ\psi\vert_E=\phi\vert_E$.
\end{definition}
\begin{definition}
A unital $C^{\ast}$-algebra $A$ has lifting property (LP) if for any c.c.p. map $\phi:A\rightarrow B/J$ is liftable, for any $C^{\ast}$-algebra $B$, and for any closed ideal $J$. A unital $C^{\ast}$-algebra $A$ has locally lifting property (LLP) if any c.c.p. map $\phi:A\rightarrow B/J$ is locally liftable for any $C^{\ast}$-algebra $B$ and closed ideal $J$ of $B$.
\end{definition}

\begin{definition} A unital $C^{\ast}$-algebra $A$ is called RFD if given any non-zero $a\in A$, there exists a finite dimensional representation $\pi_a:A\rightarrow M_{k(a)}(\mathbb{C})$ such that $\pi_a(a)\neq 0$.
\end{definition}

\begin{definition} A $C^{\ast}$-algebra $A$ is said to be primitive $C^{\ast}$-algebra if there exists a faithful, irreducible ${\ast}$-representation $\pi:A\rightarrow B(H)$, for some Hilbert space $H$.
\end{definition}

\begin{definition}
A separable C*-algebra $A$ satisfies the UCT (Universal Coefficient Theorem) if the natural short exact sequence
$$0\rightarrow {\rm Ext}_\Z^1({\rm K}_*(A), K_*(B))\rightarrow K^*(A,B)\overset{\gamma}{\rightarrow}{\rm Hom}( K_*(A),K_*(B))\rightarrow 0$$
is exact for any separable C*-algebra $B$, where $\gamma$ is the dual of the $KK$-pairing $$K_*(A)\ot KK_*(A,B)\rightarrow K_*(B).$$
\end{definition}

\begin{proposition}\label{PropLifting}
Let $A, B$ be unital C*-algebras and assume that $B$ is finite dimensional.
\begin{enumerate}
    \item $A$ has LP (resp. LLP) if and only if $\U(A,B)$ has LP (resp. LLP).
    \item $A$ is RFD if and only if $\U(A,B)$ is RFD.
    \item If $A$ is RFD and $({\rm dim}(A)-1)({\rm dim}(B)-1)\geq 2$ then $\U(A,B)$ is primitive.
    \item If $A$ satisfies the UCT then $\Ucal(A,B)$ satisfies the UCT.
\end{enumerate}
\end{proposition}
\begin{proof} For notational ease, we denote $\rho:=\rho_{A,B}$ and $\U:=\U(A,B)$.

\vspace{0.2cm}

\noindent$(1)$. Suppose that $A$ has LP (resp. LLP) It is known that LP and LLP are preserved under full free products (see \cite{Bo97} for LP and \cite{Pi96} for LLP). Hence, by Lemma \ref{RmkFree}, it suffices to prove that $\U$ has LP (resp. LLP) for $B=M_N(\C)$. In that case, $M_N(\C)\ot\U\simeq A*M_N(\C)$ has LP (resp. LLP), whence it follows that $\U$ has LP (resp. LLP).

\vspace{0.2cm}
\noindent Conversely, assume that $\U$ has LLP. Let $\pi\,:\,\U\rightarrow A$ be the unique unital $*$-homomorphism such that $(\id\ot\pi)\rho(a)=1_B\ot a$, for all $a\in A$. Let $\varphi\,:\,A\rightarrow C/J$ be a ucp map, $E\subset A$ a finite dimensional operator system  and write $p\,:\,C\rightarrow C/J$ the canonical surjection. Consider the ucp map $\id_B\ot\varphi\pi\,:\,B\ot\U\rightarrow B\ot(C/J)=(B\ot C)/(B\ot J)$. Since $B\ot\U$ has LLP, we find a ucp map $\theta\,:\,\rho(E)\rightarrow B\ot C$ such that $(\id_B\ot p)\theta=(\id_B\ot\varphi\pi)\vert_{\rho(E)}$. Fix any state $\omega\in B^*$ and define the ucp map $\psi=(\omega\ot\id)\theta\rho\,:\,E\rightarrow C$. Then, it is easy to check that $p\psi=\varphi\vert_E$. The proof in the LP case is similar.

\vspace{0.2cm}

\noindent$(2)$. Assume that $A$ is RFD. It is known that RFD is preserved under full free products \cite{MR1168356} so, as in $(1)$, we may and will assume that $B=M_N(C)$ so that $M_N(\C)\ot\U\simeq A*M_N(\C)$ is RFD which implies that $\U$ itself is RFD. Conversely, if $\U$ is RFD then $B\ot\U$ also is and since $\rho$ is faithful, we deduce that $A$ is RFD.
\vspace{0.2cm}

\noindent$(3)$. It is shown in \cite{DyTo14} that a full free product of RFD unital C*-algebras $A_1$, $A_2$ with $({\rm dim}(A_1)-1)({\rm dim}(A_2)-1)\geq 2$ is primitive. First assume that ${\rm dim}(B)=2$ so that $B=\C^2$, ${\rm dim}(A)\geq 3$ and $\U(A,B)=A*A$ which is primitive. Assume now that ${\rm dim}(B)\geq 3$. If $B$ is abelian, write $B=\C^L$ with $L\geq 3$. Then ${\rm dim}(A)\geq 2$ and $\Ucal(A,B)=(A*A)*A^{*(L-2)}$, with ${\rm dim}(A)\geq 2$ so ${\rm dim}(A*A)= \infty$ hence $\U(A,B)$ is again primitive. If $B$ is not abelian, write $B=\C^L\oplus B_0$, where $B_0=\bigoplus_{\kappa=1}^KM_{N_\kappa}(\C)$, with $K\geq 1$ and each $N_\kappa\geq 2$. If $L=0$, we have $\U(A,B)=*_{\kappa=1}^K\U(A,M_{N_\kappa}(\C))$. Note that, whenever $N\geq 2$ and ${\rm dim}(A)\geq 2$ we have ${\rm dim}(\U(A,M_{N}(\C))=\infty$ so $\U(A,B)$ is primitive whenever $K\geq 2$. If $K=1$ i.e $B=M_N(\C)$, $A*M_N(\C)=M_N(\C)\ot\U(A,M_N(\C))$ is primitive. Therefore, $M_N(\C)\ot\U(A,M_N(\C))$ is prime. Consequently, $\U(A,M_N(\C))$ is prime. Since $\U(A,M_N(\C))$ is separable, it follows that $\U(A,M_N(\C))$ is primitive. If $L\geq 1$ then $\U(A,B)=A^{*L}*\U(A,B_0)$ is also primitive, as follows from the primitivity of full free products, as mentioned above, since ${\rm dim}(A)\geq 2$ and ${\rm dim}(\U(A,B_0))=\infty$.

\vspace{0.2cm}

\noindent$(4)$. Since the UCT is stable under free products \cite[Proposition 2.2]{RS22} we may and will assume, by Lemma \ref{RmkFree}, that $B=M_N(\C)$. In that case, by Remark \ref{RmkIso}, $M_N(\C)\ot\Ucal(A,M_N(\C))\simeq A* M_N(\C)$. The stability of UCT under free products implies that $M_N(\C)\ot\Ucal(A,M_N(\C))$ satisfies the UCT. Since $M_N(\C)\ot\Ucal(A,M_N(\C))$ is Morita equivalent to $\Ucal(A,M_N(\C))$ and since the UCT is stable under KK-equivalence, we deduce that $\Ucal(A,M_N(\C))$ satisfies the UCT.
\end{proof}

\begin{remark}
Concerning the assumptions in statement $(3)$, let us note that:
\begin{itemize}
    \item If both $A=B=\C^2$ then $\U(A,B)=\C^2*\C^2$ is not primitive \cite{DyTo14}.
    \item $\U(C^*(\Z),M_N(\C))$ is primitive for $N\geq 2$ but $C^*(\Z)$ is not primitive since it is abelian.
\end{itemize}
\end{remark}

\noindent We can also deduce $K$-theory computations. For simplicity, we write $\U_N(A):=\U(A,M_N(\C))$.

\begin{corollary}\label{CorKtheory1} For all $N\geq 1$, $\U_N(A)$ and $A* M_N(\C)$ are Morita equivalent. In particular: $$K_0(\U_N(A))=\frac{K_0(A)\oplus \Z}{\mathbb{Z}([1_A]-Nx)}\quad\text{and}\quad K_1(\U_N(A))=K_1(A),$$
where $x$ is the canonical generator of $\Z$.
\end{corollary}

\begin{proof}
The statement about Morita equivalent has already been observed in the proof of Proposition \ref{PropLifting} $(4)$. Hence, $K_i(\U_N(A))=K_i(A*M_N(\C))$. Using Corollary 4.12 in \cite{FG18}, we have a six-terms exact sequence 
\[
\begin{tikzcd}
	K_0(\mathbb{C}) \ar{r} & K_0(A)\oplus K_0(M_N(\mathbb{C})) \ar{r} & K_0(A\ast M_N(\mathbb{C})) \ar{d} \\
	K_1(A\ast M_N(\mathbb{C})) \ar{u} & K_1(A)\oplus K_1(M_N(\mathbb{C})) \ar{l} & K_1(\mathbb{C}) \ar{l}
\end{tikzcd}
\]
Since $K_0(\mathbb{C})=K_0(M_N(\mathbb{C}))=\mathbb{Z}[e_{11}]$ and $K_1(\mathbb{C})=K_1(M_N(\mathbb{C}))=0$, we thus have a group isomorphism $K_0(A\ast M_N(\mathbb{C}))=\frac{K_0(A)\oplus K_0(M_N(\C))}{\mathbb{Z}([1_A]-[1_{M_N(\C)}])}=\frac{K_0(A)\oplus \Z[e_{11}]}{\mathbb{Z}([1_A]-N[e_{11}])}$.

\vspace{0.2cm}

\noindent Further, as the map $K_0(\C)=\mathbb{Z}\rightarrow K_0(A)\oplus\Z$, $k\mapsto([1_A]-N[e_{11}])$ is clearly injective, we deduce that $K_1(A*M_N(\C))=K_1(A)$.
\end{proof}

\section{Reduced operator algebras of universal quantum homomorphisms}\label{Section4}

\noindent Given a unital C*-algebra $A$ with a state $\omega\in A^*$ and a finite dimensional C*-algebra $B$ with a state $\mu\in B^*$ we introduce in this section a natural state $\widetilde{\omega}\in\Ucal(A,B)^*$ and study the C*-algebra and the von Neumann algebra obtained by the GNS construction of $\widetilde{\omega}$.

\subsection{The reduced C*-algebra}

\subsubsection{Preliminaries}

Given two unital C*-algebras $A_k$ with states $\omega_k\in A_k^*$, $k=1,2$, we recall that there exists a unique state $\omega_1*\omega_2$ one the full free product such that $\omega_1*\omega_2(x)=0$ for all $x\in A_1*A_2$ which can be written as an alternating product of elements from $A_1^\circ:=\ker(\omega_1)$ and $A_2^\circ:=\ker(\omega_2)$. The state $\omega_1*\omega_2$ is called the \textit{free product state}. When $\omega_1$ and $\omega_2$ are both GNS-faithful, the C*-algebra generated in the GNS-construction of $\omega_1*\omega_2$ is denoted by $A_1\underset{r}{*}A_2$ and is called the Voiculescu's reduced free product with respect to states $\omega_1$, $\omega_2$. Note that if $\omega_1$ or $\omega_2$ are not supposed to be GNS-faithful then the C*-algebra given by the GNS-construction of $\omega_1*\omega_2$ is the Voiculescu's reduced free product $(A_1)_{\omega_1}\underset{r}{*}(A_2)_{\omega_2}$ with respect to the (GNS-faithful) states $\omega_1$, $\omega_2$ where, given a state $\omega$ on the unital C*-algebra $A$, we denote by $A_\omega$ the C*-algebra obtained by GNS-construction of $\omega$ and we view $\omega$ as a GNS faithful state on $A_\omega$.

\vspace{0.2cm}

\noindent We will need the following Lemma.

\begin{lemma}\label{LemDualState}
Let $\omega\in A^*$ and $\mu\in M_N(\C)^*$ be states.  Then, for all $b\in M_N(\C)$ and all $x\in M_N(\C)'\cap(A*M_N(\C))$ one has $(\omega*\mu)(bx)=\mu(b)(\omega*\mu)(x)$.
\end{lemma}

\begin{proof}
Let $Q\in M_N(\C)$ be a positive matrix such that $\mu={\rm Tr}(Q\cdot)$ and take $(e_{ij})_{ij}$ matrix units in $M_N(\C)$ diagonalizing $Q$ i.e. there exists positive $\lambda_r$ with $\sum\lambda_r=1$ and $Qe_{ij}=\lambda_ie_{ij}$. For $a\in A$ define $\rho_{ij}(a):=\sum_re_{ri}ae_{jr}\in A*M_N(\C)$. By the proof of Proposition \ref{PropMNcase}, $\Ucal:=M_N(\C)'\cap(A*M_N(\C))$ is the sub-C*-algebra of $A*M_N(\C)$ generated by:
$$\{\rho_{ij}(a)\,:\,a\in A,\,1\leq i,j\leq N\}.$$
Note that the linear span of $\{\rho_{i_1j_1}(a_1)\dots\rho_{i_nj_n}(a_n)\,:\,n\geq 1,\,1\leq i_s,j_s\leq N,\,a_s\in A\}$ is a unital $*$-subalgebra of $\mathcal{U}$, hence it is dense. To ease the notation we write $\omegatilde:=\omega*\mu$. It suffices to prove that $\widetilde{\omega}(e_{st}X)=\mu(e_{st})\widetilde{\omega}(X)$ for $X=\rho_{i_1j_1}(a_1)\dots\rho_{i_nj_n}(a_n)$. Note that
$$X=\sum_{r=1}^Ne_{ri_1}a_1e_{j_1i_2}a_2\dots a_ne_{j_nr}$$
so that $e_{st}X=e_{si_1}a_1e_{j_1i_2}a_2\dots a_ne_{j_nt}$. Let $E\,:\,A*M_N(\C)\rightarrow M_N(\C)$ be the canonical conditional expectation onto $M_N(\C)$. We have:
$$
\widetilde{\omega}(e_{st}X)=\mu\circ E(e_{si_1}a_1e_{j_1i_2}a_2\dots a_ne_{j_nt})=\mu(e_{si_1}E(X')e_{j_nt}),
$$
where $X':=a_1e_{j_1i_2}a_2\dots a_n$. Hence,
\begin{eqnarray*}
\widetilde{\omega}(e_{st}X)&=&{\rm Tr}(Qe_{si_1}E(X')e_{j_nt})
={\rm Tr}(e_{j_nt}Qe_{si_1}E(X'))=\delta_{s,t}\lambda_s{\rm Tr}(e_{j_ni_1}E(X'))\\
&=&\mu(e_{st}){\rm Tr}(e_{j_ni_1}E(X')).
\end{eqnarray*}
Writing $e_{j_ni_1}=\sum_r\lambda_re_{j_nr}e_{ri_1}$ we get:
\begin{eqnarray*}
\widetilde{\omega}(e_{st}X)&=&\mu(e_{st})\sum_r\lambda_r{\rm Tr}(e_{j_nr}e_{ri_1}E(X'))=\mu(e_{st})\sum_r\lambda_r{\rm Tr}(e_{ri_1}E(X')e_{j_nr})\\
&=&\mu(e_{st})\sum_r{\rm Tr}(Qe_{ri_1}E(X')e_{j_nr})=\mu(e_{st}){\rm Tr}(QE(X))=\mu(e_{st})\widetilde{\omega}(X).
\end{eqnarray*}
\end{proof}

\subsubsection{The one-block case}
We first consider the case $\U_N(A)=\U(A,M_N(\C))$. Fix states $\omega\in A^*$ and $\mu\in M_N(\C)^*$ and denote $A^\circ:={\rm ker}(\omega)$, $M_N(\C)^\circ:={\rm ker}(\mu)$. Let $\rho\,:\, A\rightarrow M_N(\C)\ot\U_N(A)$ be the universal quantum homomorphism from $A$ to $M_N(\C)$. An operator $X\in M_N(\C)\ot\Ucal_N(A)$ will be called \textit{reduced} if it is an alternating product of elements in $\rho(A^\circ)$ and in $M_N(\C)^\circ\ot 1$. The number of alternating factors in $X$ will be called the \textit{length} of $X$.

\vspace{0.2cm}

\noindent The following Proposition is a direct consequence of Lemma \ref{LemDualState} and Proposition \ref{PropMNcase}.

\begin{proposition}\label{PropState}
There exists a unique state $\widetilde{\omega}\in\Ucal_N(A)^*$ such that $(\mu\ot\widetilde{\omega})(X)=0$ for all reduced operator $X\in M_N(\C)\ot \U_N(A)$ and $(\mu\ot\widetilde{\omega})\circ\rho=\omega$.
\end{proposition}

\noindent Let $\U_{\omega,\mu}(A)$ be the C*-algebra generated by the GNS construction of $\widetilde{\omega}$ i.e. $\U_{\omega,\mu}(A):=\pi_{\widetilde{\omega}}(\U_N(A))$, where $(H_{\widetilde{\omega}},\pi_{\widetilde{\omega}},\xi_{\widetilde{\omega}})$ is the GNS construction of $\widetilde{\omega}$. We still denote by $\widetilde{\omega}\in\U_{\omega,\mu}(A)^*$ the GNS-faithful state defined by  $\widetilde{\omega}:=\langle\cdot\xi_{\widetilde{\omega}},\xi_{\widetilde{\omega}}\rangle$. Write $\widetilde{\rho}:=(\id\ot\pi_{\widetilde{\omega}})\rho\,:\,A\rightarrow M_N(\C)\ot\U_{\omega,\mu}(A)$ and note that $\U_{\omega,\mu}(A)$ is generated by the coefficients of $\widetilde{\rho}$ and we have $(\mu\ot\widetilde{\omega})\circ\widetilde{\rho}=\omega$ and $(\mu\ot\widetilde{\omega})((\id\ot\pi_{\omegatilde})(X))=0$ for every reduced operator $X\in M_N(\C)\ot\U_N(A)$.

\vspace{0.2cm}

\noindent Let $(H_\omega,\pi_\omega,\xi_\omega)$ be the GNS-construction of $\omega$ and $A_\omega:=\pi_\omega(A)$. We still denote by $\omega$ the GNS-faithful state on $A_\omega$ defined by $\omega:=\langle\cdot\xi_\omega,\xi_\omega\rangle$. Note that any state on $M_N(\C)$ is GNS-faithful.

\vspace{0.2cm}

\noindent Recall that a state $\omega\in A^*$ is called \textit{diffuse} if its normal extension to $A^{**}$, still denoted $\omega$, satisfies $\omega(p)=0$ for all minimal projection $p\in A^{**}$.

\begin{theorem}\label{THMReducedMN}
There exists a unique isomorphism $\psi_\omega\,:\,A_\omega\underset{r}{*}M_N(\C)\rightarrow M_N(\C)\ot\UNAomega$ such that $\psi_\omega\circ\pi_\omega=\rhotilde$ and $\psi_\omega(b)=b\ot 1$ for all $b\in M_N(\C)$, where the reduced free product is with respect to the GNS-faithful states $\mu\in M_N(\C)^*$ and $\omega\in (A_\omega)^*$. Moreover,
\begin{enumerate}
\item The C*-algebras $\U_{\omega,\mu}(A)$ and $\U_N(A_\omega)$ are $KK$-equivalent.
    \item If $A$ is exact then $\U_{\omega,\mu}(A)$ is exact. The converse holds when $\omega$ is GNS-faithful.
    \item Suppose that $\omega$ is GNS-faithful. If $\U_{\omega,\mu}(A)$ is nuclear then $A$ is nuclear and the converse holds if $\omega$ or $\mu$ is a pure state. If both $\omega$ and $\mu$ are pure then $\omegatilde$ is pure.
    \item If $N\geq 2$, $\omega$ and $\mu$ are faithful traces on $A$ and $M_N(\C)$ respectively and $\omega$ is diffuse then $\U_{\omega,\mu}(A)$ is simple with unique trace and has stable rank $1$.
    \item If $\omega$ and $\mu$ are faithful traces on $A$ and $M_N(\C)$ respectively, $N\geq 2$ and $\omega_*(K_0(A))\subset\frac{1}{N}\Z$, where $\omega_*\,:\,K_0(A)\rightarrow\R$ is the group homomorphism induced by $\omega$, then $\U_{\omega,\mu}(A)$ has no non-trivial projection.
\end{enumerate}
\end{theorem}

\begin{proof}
Consider the isomorphism $\psi\,:\,A*M_N(\C)\rightarrow M_N(\C)\ot\UNA$ from Remark \ref{RmkIso} so that we get a surjective $*$-homomorphism $\psi_\omega:=(\id\ot\pi_{\omegatilde})\psi\,:\,A*M_N(\C)\rightarrow M_N(\C)\ot\UNAomega$. The state $\mu\ot\omegatilde\in (M_N(\C)\ot\UNAomega)^*$ is GNS-faithful and, by Proposition \ref{PropState}, $\psi_\omega$ intertwines the free product state $\omega*\mu\in (A*M_N(\C))^*$ and the state $\mu\ot\omegatilde\in (M_N(\C)\ot\UNAomega)^*$ so $M_N(\C)\ot\UNAomega$ is the C*-algebra of the GNS-construction of the free product state $\omega*\mu$ and $\psi_\omega$ factorizes to an isomorphism $\psi_\omega\,:\,A_\omega\underset{r}{*}M_N(\C)\rightarrow M_N(\C)\ot\UNAomega$ satisfying the claimed properties. The uniqueness is clear.

\vspace{0.2cm}

\noindent(1). The $KK$-equivalence is obtained as the composition (i.e. Kasparov product) of the following $KK$-equivalences. First, $\U_N(A_\omega)$ is Morita equivalent to $M_N(\C)\ot\U_N(A_\omega)$ which is isomorphic to $A_\omega*M_N(\C)$. Then, from \cite{FG20}, the canonical surjection $\pi\,:\, A_\omega*M_N(\C)\rightarrow A_\omega\underset{r}{*} M_N(\C)$ is known to be a $KK$-equivalence. Next, $A_\omega\underset{r}{*} M_N(\C)$ is isomorphic to $M_N(\C)\ot\U_{\omega,\mu}(A)$ which is Morita equivalent to $\U_{\omega,\mu}(A)$.

\vspace{0.2cm}

\noindent$(2)$. We first show the equivalence when $\omega$ is GNS-faithful. Note that, since a subalgebra of an exact C*-algebra is exact, the result of \cite{Dy04} implies that $A$ is exact if and only if $A\underset{r}{*} M_N(\C)$ is exact. Moreover, for any C*-algebra $D$, it is clear that $D$ is exact if and only if $M_N(\C)\ot D$ is exact. Since $A\underset{r}{*} M_N(\C)\simeq M_N(\C)\ot\U_{\omega,\mu}(A)$, this concludes the proof of $(2)$ in the GNS faithful case. Take now any state $\omega\in A^*$ and assume that $A$ is exact so that $A_\omega$ also is being of quotient of $A$. The state $\omega$ viewed on $A_\omega$ being GNS-faithful, we can conclude from the first part of the proof and since we clearly have $\UNAomega=\U_{\omega,\mu}(A_\omega)$, we conclude that $\UNAomega$ is exact.

\vspace{0.2cm}

\noindent$(3)$. We first note that, while nuclearity does not pass to subalgebras, it is easy to check that if $C\subset D$ is a subalgebra with a conditional expectation $E\,:\,D\rightarrow C$ then $D$ nuclear implies that $C$ is nuclear.

\vspace{0.2cm}

\noindent Fix $\omega\in A^*$ a GNS-faithful state and assume that either $\omega$ or $\mu$ is pure. It is known that \cite[Theorem 4.8.7]{BO08} $A\underset{r}{*}M_N(\C)\simeq M_N(\C)\ot\U_{\omega,\mu}(A)$ is nuclear, whenever $A$ is nuclear. Since the subalgebra $\U_{\omega,\mu}(A)\subset M_N(\C)\ot\U_{\omega,\mu}(A)$ has a conditional expectation we deduce that $\U_{\omega,\mu}(A)$ is nuclear. Conversely, if $\U_{\omega,\mu}(A)$ is nuclear then, since $M_N(\C)$ is finite dimensional, $M_N(\C)\ot \U_{\omega,\mu}(A)\simeq A\underset{r}{*}M_N(\C)$ also is and because the subalgebra $A\subset A\underset{r}{*}M_N(\C)$ has a conditional expectation, we deduce that $A$ is nuclear.

\vspace{0.2cm}

\noindent Assume now that $\omega$ and $\mu$ are both pure (and $\omega$ is still GNS-faithful). By Voiculescu's characterisation of the commutant of the reduced free product, we deduce that the free product state $\omega*\mu\in(A*M_N(\C))^*$ is pure. Using the state-preserving isomorphism $\psi_\omega$, we deduce that $\mu\ot\omegatilde\in (M_N(\C)\ot\UNAomega)^*$ is pure. A direct convexity argument implies that $\omegatilde$ itself is pure.

\vspace{0.2cm}

\noindent$(4)$. We use \cite[Theorem 9.1]{Th22} and the fact that, for any unital C*-algebra $D$, $M_N(\C)\ot D$ simple with unique trace and stable rank $1$ implies $D$ simple with unique trace and stable rank $1$ (for the stable rank statement see \cite[Theorem 3.3]{Ri83}).

\vspace{0.2cm}

\noindent$(5)$. Note that $\mu$ is the unique tracial state on $M_N(\C)$ and the free product state $\chi:=\omega*\mu$ is a faithful trace on $A\underset{r}{*}M_N(\C)$. Using the exact sequence for the K-theory of reduced free products from \cite{FG20} we get $\chi_*(K_0(A\underset{r}{*}M_N(\C)))=\mu_*(K_0(M_N(\C)))+\omega_*(K_0(A))\subset\frac{1}{N}\Z+\frac{1}{N}\Z=\frac{1}{N}\Z$. Let $p\in \UNAomega$ be a projection. Since $\psi_\omega$ intertwines $\chi$ with ${\rm tr}\ot\omegatilde$ we have: $$\chi(\psi_\omega^{-1}( e_{11}\ot p))={\rm tr}(e_{11})\omegatilde(p)=\frac{1}{N}\omegatilde(p).$$
Since $\psi_\omega^{-1}( e_{11}\ot p)$ is a projection, we have $\chi(\psi_\omega^{-1}( e_{11}\ot p))\in\chi_*(K_0(A\underset{r}{*}M_N(\C)))\subset\frac{1}{N}\Z$.
It follows that $\omegatilde(p)\in\Z$ hence, $\omegatilde(p)\in\{0,1\}$. We conclude by faithfulness of $\omegatilde$.
\end{proof}

\subsubsection{The finite dimensional case}\label{SectionCstarFinitedim}

Assume that $B$ is finite dimensional and write
$$B=\bigoplus_{\kappa=1}^KM_{N_\kappa}(\C)$$
with canonical surjective unital $*$-homomorphisms $\chi_\kappa\,:\, B\rightarrow M_{N_\kappa}(\C)$. Let $A$ be any unital C*-algebra and write $\rho\,:\, A\rightarrow B\ot\Ucal(A,B)$ the universal quantum homomorphism from $A$ to $B$. It follows from Lemma \ref{RmkFree} that $\U(A,B)=\underset{\kappa=1}{\overset{K}{*}}\Ucal_{N_\kappa}(A)$.

\vspace{0.2cm}

\noindent Fix states $\omega\in A^*$, $\mu\in B^*$ and assume that $\mu$ is GNS-faithful. Write $\mu=\sum_{k=1}^N{\rm Tr}_\kappa(Q_\kappa\cdot)$, where ${\rm Tr}_\kappa$ is the unique trace on $M_{N_\kappa}(\C)$ such that ${\rm Tr}_\kappa(1)=N_\kappa$ and $Q_\kappa\in M_{N_\kappa}(\C)$ is positive non-zero. Define the state $\mu_\kappa:=\frac{1}{{\rm Tr}_\kappa(Q_\kappa)}{\rm Tr}_\kappa(Q_\kappa\cdot)\in M_{N_\kappa}(\C)^*$. Write $M_{N_\kappa}(\C)^\circ:={\rm ker}(\mu_k)$. Let $\widetilde{\omega}_\kappa\in\Ucal_{N_\kappa}(A)^*$ the state constructed in the one-block case from $\omega$ and $\mu_\kappa$ and write $\Ucal_{N_\kappa}(A)^\circ:={\rm ker}(\widetilde{\omega}_\kappa)$. Note that $\widetilde{\omega}_\kappa$ is tracial whenever $\omega$ and $\mu$ are.

\vspace{0.2cm}

\noindent Let us call an operator $x\in \Ucal(A,B)$ \textit{reduced} whenever $x$ is of the form $x=x_1\dots x_n$, for $n\geq 1$, $x_s\in \Ucal_{N_{\kappa_s}}(A)^\circ$ and $\kappa_s\neq\kappa_{s+1}$. From the free product decomposition of $\Ucal(A,B)$, there exists a unique state $\omegatilde\in\Ucal(A,B)^*$ such that $\omegatilde(x)=0$ for any reduced operator $x$.

\vspace{0.2cm}

\noindent The \textit{reduced} C*-algebra of quantum homomorphisms from $A$ to $B$ is the C*-algebra obtained by the GNS-construction of $\omegatilde$ and it is denoted by $\Ucal_{\omega,\mu}(A,B)$. By construction, $\Ucal_{\omega,\mu}(A,B)$ is the Voiculescu's reduced free product of the $\Ucal_{\omega,\mu_\kappa}(A)$, for $1\leq \kappa\leq K$, with respect to the GNS-faithful states $\omegatilde_\kappa\in\Ucal_{\omega,\mu_\kappa}(A)^*$. Let $N:=\sum_{\kappa=1}^KN_\kappa$.

\begin{theorem}
The following holds.
\begin{enumerate}
    \item If $A$ is exact then $\Ucal_{\omega,\mu}(A,B)$ is exact and the converse holds when $\omega$ is GNS-faithful.
    \item Suppose that $\omega$ is GNS-faithful. If  $\U_{\omega,\mu}(A)$ is nuclear then $A$ is nuclear and the converse holds if both $\omega$ and $\mu$ are pure states. In that case, $\omegatilde\in\U_{\omega,\mu}(A,B)^*$ is also pure.
    \item If $\omega$ and $\mu$ are faithful traces, $\omega$ is diffuse and $N\geq 2$ then $\Ucal_{\omega,\mu}(A,B)$ is simple with unique trace and has stable rank $1$.
    
\end{enumerate}
\end{theorem}

\begin{proof}
$(1)$. Since $\Ucal_{\omega,\mu}(A,B)$ is the Voiculescu's reduced amalgamated free product of the $\Ucal_{\omega,\mu_\kappa}(A)$, for $1\leq\kappa\leq K$, with respect to the GNS-faithful states $\widetilde{\omega}_\kappa$, we can argue exactly as in the proof of Theorem \ref{THMReducedMN}, assertion $(2)$ by using assertion $(2)$ of Theorem \ref{THMReducedMN} and \cite{Dy04}.

\vspace{0.2cm}

\noindent$(2)$. Follows from Theorem \ref{THMReducedMN} and \cite[Theorem 4.8.7]{BO08} by induction on $K$.

\vspace{0.2cm}

\noindent$(3)$. If $K=1$, this is Theorem \ref{THMReducedMN}, assertion $(4)$. If $K\geq 2$ write $\Ucal_{\omega,\mu}(A,B)=D\underset{r}{*}\Ucal_{\omega,\mu_K}(A)$, where $D\neq\C$. Since $\omega$ is diffuse, $A$ has to be infinite dimensional and so is $D$. Here, the reduced free product is with respect to $\tau$ and $\widetilde{\omega}_K$ which are both faithful tracial states. If $N_K=1$ then $\Ucal_{\omega,\mu_K}(A)=A$ and $\omegatilde_K=\omega$ is diffuse and if $N_K\geq 2$ then $\Ucal_{\omega,\mu_K}(A)$ is simple by Theorem \ref{THMReducedMN} so, by \cite[Corollary 5.6]{Th22}, the trace $\widetilde{\omega}_K$ is diffuse. In both cases, we can apply \cite[Theorem 9.1]{Th22}.
\end{proof}

\subsection{The von Neumann algebra}

\noindent Let $\omega$ be a faithful state on $A$ and view $A\subset\Lcal({\rm L}^2(A,\omega))$, where $({\rm L}^2(A,\omega),\id,\xi_\omega)$ is the GNS construction of $\omega$. Define $A''\subset\Lcal({\rm L}^2(A,\omega))$ as the von Neumann subalgebra of $\Lcal({\rm L}^2(A,\omega))$ generated by $A$ and $\omega''$ be the normal faithful state on $A''$ defined by $\omega''(x):=\langle x \xi_\omega,\xi_\omega\rangle$.

\vspace{0.2cm}

\noindent By \cite[Lemma 3.2]{Th22}, the state $\omega$ on $A$ is diffuse if and only if the von Neumann algebra $A''\subset\Lcal({\rm L}^2(A,\omega))$ is diffuse.

\subsubsection{The one-block case}

\noindent Fix a faithful state $\mu\in M_N(\C)^*$ and define the von Neumann algebra $\U_{\omega,\mu}(A)''$ as the von Neumann algebra generated by $\UNAomega$ in the GNS-construction of $\omegatilde$. We still denote by $\widetilde{\omega}$ the canonical normal faithful state on $\U_{\omega,\mu}(A)''$ extending $\widetilde{\omega}$ on $\UNAomega$.

\begin{remark}
For $Q\in M_N(\C)$ a strictly positive matrix, we use the notation ${\rm Sd}(Q)\subset\R^*_+$ for the closed multiplicative subgroup of $\R^*_+$ generated by $\{\frac{r}{s}\,:\,r,s\in{\rm Sp}(Q)\}$, where ${\rm Sp}(Q)$ denotes the spectrum of $Q$. We also write, for $A\subset\R^*_+$, $A^\perp:=\{t\in\R\,:\,a^{it}=1\text{ for all }a\in A\}$. Note that $A^\perp=\overline{\langle A\rangle}^\perp$, where $\overline{\langle A\rangle}$ denotes the closed subgroup of $\R^*_+$ generated by $A$. Finally, for $A\subset\R^*_+$, we use the notation $\tau(A)$ for the smallest topology on $\R$ making the maps $(t\mapsto a^{it})$ continuous for all $a\in A$. Note that, for any topology $\tau$ on $\R$, the set for $a\in\R^*_+$ such that $(t\mapsto a^{it})$ is $\tau$-continuous is a closed subgroup of $\R^*_+$ (for the usual topology on $\R^*_+$). Hence, $\tau(A)=\tau(\overline{\langle A\rangle})$. We use the notation $\tau(Q):=\tau({\rm Sd}(Q))$.
\end{remark}

\begin{theorem}\label{VNMN}
Let $\omega\in A^*$, $\mu={\rm Tr}(Q\cdot)\in M_N(\C)^*$ be faithful states. There exists a unique unital normal $*$-isomorphism $\psi_\omega''\,:\,A''*M_N(\C)\rightarrow M_N(\C)\ot\U_{\omega,\mu}(A)''$ extending $\psi_\omega$. Moreover,
\begin{enumerate}
\item If $\omega$ is diffuse and $N\geq 2$ then $\U_{\omega,\mu}(A)''$ is a non-amenable, full and prime factor of type ${\rm II}_1$ or of type ${\rm III}_\lambda$ with $\lambda\neq 0$ and $T(\UNAomega'')=\{t\in\R\,:\,\sigma_t^{\omega''}=\id\}\cap{\rm Sd(Q)}^\perp$. Its Connes' $\tau$-invariant is the smallest topology on $\R$ containing $\tau(Q)$ and making continuous the map $(t\mapsto\sigma_t^{\omega''})\,:\,\R\rightarrow{\rm Aut}(A'')$.
\item If ${\rm dim}(A)\geq 2$ and $N\geq 2$ then $\UNAomega''$ has no Cartan subalgebra.
    \item $A''$ has the Haagerup property if and only if $\U_{\omega,\mu}(A)''$ has the Haagerup property.
    \item If $\omega$ and $\mu$ are both traces then $A''$ is Connes embeddable if and only if $\U_{\omega,\mu}(A)''$ is Connes embeddable.
\end{enumerate}
\end{theorem}

\begin{proof}
Since the isomorphism $\psi_\omega\,:\,A\underset{r}{*} M_N(\C)\rightarrow M_N(\C)\ot\UNAomega$ from Theorem \ref{THMReducedMN} intertwines the free product state $\omega*\mu$ and the state $\mu\ot\omegatilde$, $\psi_\omega$ extends to a spacial isomorphism $\psi_\omega''$ between the von Neumann algebras generated in the respective GNS constructions.
\vspace{0.2cm}

\noindent$(1)$. Since $A''$ is diffuse and $N\geq 2$, a direct application of the results of \cite[Theorems 3.4 and 3.7]{Ue11a} gives that the von Neumann algebraic free product $A''*M_N(\C)$ is a full factor  of type ${\rm II}_1$ or of type ${\rm III}_\lambda$ with $\lambda\neq 0$ (hence non-amenable) and it's Connes' $T$ invariant is given by $T(A''*M_N(\C))=\{t\,:\,\sigma_t^{\omega''}=\id=\sigma_t^\mu\}$. Since $M_N(\C)\ot\UNAomega''\simeq A''*M_N(\C)$, $\UNAomega''$ also is a non-amenable and full factor of type ${\rm II}_1$ of ${\rm III}_\lambda$ with $\lambda\neq 0$ and $T(\UNAomega'')=\{t\,:\,\sigma_t^{\omega''}=\id=\sigma_t^\mu\}$. A direct computation shows that, whenever $(e_{kl})_{kl}$ are matrix units in $M_N(\C)$ diagonalizing $Q$ i.e. $Qe_{ij}=\lambda_ie_{ij}$ then, $\sigma_t^\mu(e_{kl})=\left(\frac{\lambda_k}{\lambda_l}\right)^{it}e_{kl}$ for all $t\in\R$. Hence $\sigma_t^\mu=\id\Leftrightarrow t\in\{\frac{r}{s}\,:\,r,s\in{\rm Sp(Q)}\}^\perp={\rm Sd(Q)}^\perp$. This shows the formula for the $T$-invariant. By \cite[Theorem 3.2]{Ue11b}, the $\tau$-invariant of $A''*M_N(\C)$ is the smallest topology on $\R$ making both maps $(t\mapsto\sigma_t^{\omega''})$ and $(t\mapsto\sigma_t^\mu)$ continuous. But from \cite[Corollary B]{HMV19}, the $\tau$-invariant of $M_N(\C)\ot\UNAomega''$ is the $\tau$-invariant of $\UNAomega''$ (this also can be easily checked directly). To conclude the proof of the statement on the $\tau$-invariant, it suffices to note that, by the explicit formula for $\sigma_t^\mu$, one has $(t\mapsto\sigma^\mu_t)$ is continuous if and only if $(t\mapsto \left(\frac{r}{s}\right)^{it})$ is continuous for all $r,s\in{\rm Sp}(Q)$.

\vspace{0.2cm}

\noindent$(2)$ follows from \cite[Theorem A]{BHR14}, $(3)$ follows from from \cite[Corollary C]{CKSVW21} and $(4)$ follows from \cite{BDJ08}.\end{proof}

\noindent Write $\rho'':=\psi_\omega''\vert_{A''}\,:\,A''\rightarrow M_N(\C)\ot\UNAomega''$ and note that $\UNAomega''$ is the von Neumann algebra generated by the coefficients of $\rho''$. 

\begin{lemma}\label{LemmaModular}
Let $(e_k)_k$ be an orthonormal basis of $\C^N$ diagonalizing $Q$ i.e. $Q=\sum_{k=1}^N\lambda_ke_{kk}$, where ${\rm Sp}(Q)=\{\lambda_k\,:\,1\leq k\leq N\}$. Write, for $a\in A''$, $\rho''(a)=\sum_{k,l}e_{kl}\ot\rho''_{kl}(a)$. Then, $$\sigma_t^{\widetilde{\omega}}(\rho''_{kl}(a))=\left(\frac{\lambda_l}{\lambda_k}\right)^{it}\rho''_{kl}(\sigma_t^{\omega''}(a)).$$
Moreover, the smallest topology on $\R$ containing $\tau(Q)$ and making the map $(t\mapsto\sigma_t^{\omega''})$ continuous is the smallest topology on $\R$ making the map $(t\mapsto\sigma_t^{\widetilde{\omega}})$ continuous.
\end{lemma}

\begin{proof}
Since the isomorphism $\psi_{\omega}''\,:\,A''*M_{N}(\C)\rightarrow M_{N}(\C)\ot \Ucal_{\omega,\mu}(A)''$ of Theorem \ref{VNMN} intertwines the states $\omega''*\mu$ and $\mu\ot\widetilde{\omega}$, it also intertwines the modular groups. Hence, for all $a\in A''$,
\begin{eqnarray*}
(\sigma_t^\mu\ot\sigma_t^{\widetilde{\omega}})\rho''(a)&=&\sum_{k,l}\sigma_t^\mu(e_{kl})\ot\sigma_t^{\widetilde{\omega}}(\rho''_{kl}(a))=\sum_{k,l} e_{kl}\ot\left(\frac{\lambda_k}{\lambda_l}\right)^{it}\sigma_t^{\widetilde{\omega}}(\rho''_{kl}(a))\\
&=&\rho''(\sigma_t^{\omega''*\mu}(a))=\rho''(\sigma_t^{\omega''}(a))=\sum_{k,l}e_{kl}\ot\rho''_{kl}(\sigma_t^{\omega''}(a)).
\end{eqnarray*}
The formula of the Lemma follows. Let $\tau$ be the smallest topology on $\R$ making the map $(t\mapsto\sigma_t^{\widetilde{\omega}})$ continuous and $\tau_0$ be the smallest topology on $\R$ containing $\tau(Q)$ and for which the map $(t\mapsto\sigma_t^{\omega''})$ is continuous. Then, both maps $(t\mapsto\sigma_t^{\mu})$ and $(t\mapsto\sigma_t^{\omega''})$ are $\tau_0$-continuous so $(t\mapsto\sigma_t^{\omega''*\mu})$ is $\tau_0$-continuous. Since the isomorphism $\psi_\omega''$ intertwines $\omega''*\mu$ and $\mu\ot\widetilde{\omega}$, we deduce that $(t\mapsto\sigma_t^\mu\ot\sigma_t^{\widetilde{\omega}})$ is $\tau_0$-continuous. It implies that $(t\mapsto\sigma_t^{\widetilde{\omega}})$ is $\tau_0$ continuous so $\tau\subseteq\tau_0$. Using the explicit formula for $\sigma_t^{\widetilde{\omega}}$, we see that $(t\mapsto \left(\frac{r}{s}\right)^{it})$ is $\tau$-continuous, for all $r,s\in{\rm Sp}(Q)$. It follows that $(t\mapsto\sigma_t^\mu)$ is $\tau$-continuous. Hence, $(t\mapsto\sigma_t^\mu\ot\sigma_t^{\widetilde{\omega}})$ is $\tau$-continuous. Using that $\psi_\omega''$ intertwines $\omega''*\mu$ and $\mu\ot\widetilde{\omega}$ we deduce, as before, that $(t\mapsto\sigma_t^{\omega''})$ is $\tau$-continuous hence, $\tau_0\subseteq\tau$.
\end{proof}

\subsubsection{The finite dimensional case}
We assume that $B=\bigoplus_{\kappa=1}^KM_{N_\kappa}(\C)$ is finite dimensional and $\mu\in B^*$ is a faithful state. We define the von Neumann algebra $\Ucal_{\omega,\mu}(A,B)''$ as the von Neumann algebra generated by $\Ucal_{\omega,\mu}(A,B)$ in the GNS-construction of $\omegatilde$. By construction, $\Ucal_{\omega,\mu}(A,B)''$ is the von Neumann algebraic free product of $\Ucal_{\omega,\mu_\kappa}(A)''$, $1\leq\kappa\leq K$, with respect to the faithful normal states $\omegatilde_\kappa\in\Ucal_{\omega,\mu_\kappa}(A)''_*$.

\vspace{0.2cm}

\noindent Recall that $\mu=\sum_{\kappa=1}^K{\rm Tr}(Q_\kappa\cdot)$, where $Q_\kappa\in M_{N_\kappa}(\C)$ is strictly positive. Denote by ${\rm Sd}(\mu)$ the closed multiplicative subgroup of $\R^*_+$ generated by $\{\frac{r}{s}\,:\,r,s\in{\rm Sp}(Q_\kappa),1\leq\kappa\leq K\}$. We also write $\tau(\mu):=\tau({\rm Sd}(\mu))$ and, as before,  $N:=\sum_{\kappa=1}^KN_\kappa$.

\begin{theorem}
The following holds.
\begin{enumerate}
\item If $\omega$ is diffuse and $N\geq 2$ then $\U_{\omega,\mu}(A,B)''$ is a non-amenable, full and prime factor of type ${\rm II}_1$ or of type ${\rm III}_\lambda$ with $\lambda\neq 0$ and $$T(\U_{\omega,\mu}(A,B)'')=\{t\in\R\,:\,\sigma_t^{\omega''}=\id\}\cap{\rm Sd(\mu)}^\perp.$$
It's Connes' $\tau$-invariant is the smallest topology on $\R$ containing $\tau(\mu)$ and making continuous the map $(t\mapsto\sigma_t^{\omega''})\,:\,\R\rightarrow{\rm Aut}(A'')$.
\item If ${\rm dim}(A)\geq 2$ and either $K\neq 2$ and $N\geq 2$ or $K=2$ and ${\rm dim}(A)+N\geq 5$ then $\U_{\omega,\mu}(A,B)''$ has no Cartan subalgebra.
    \item $A''$ has the Haagerup property if and only if $\U_{\omega,\mu}(A,B)''$ has the Haagerup property.
    \item If $\omega$ and $\mu$ are both traces then $A''$ is Connes embeddable if and only if $\U_{\omega,\mu}(A,B)''$ is Connes embeddable.
\end{enumerate}
\end{theorem}

\begin{proof}
$(1)$. By Theorem \ref{VNMN} we may and will assume that $K\geq 2$ and write $\Ucal_{\omega,\mu}(A,B)''=M*\Ucal_{\omega,\mu_K}(A)''$, where the von Neumann algebraic free product is with respect to the faithful normal states $\Omega:=\omegatilde_1*\dots*\omegatilde_{K-1}\in M_*$ and $\omegatilde_K$. Note that, since $\omega$ is diffuse, $\U_{\omega,\mu_\kappa}(A)''$ is always diffuse: it follows from Theorem \ref{VNMN} for $N_\kappa\geq 2$ (but can be easily proved directly) and, if $N_\kappa=1$ then $\U_{\omega,\mu_\kappa}(A)=A''$ is diffuse hence, both $M$ and $\Ucal_{\omega,\mu_K}(A)''$ are diffuse so that we may apply \cite[Theorems 3.4 and 3.7]{Ue11a} to deduce that $\Ucal_{\omega,\mu}(A,B)''$ is a full factor of type ${\rm II}_1$ or of type ${\rm III}_\lambda$ with $\lambda\neq 0$ (hence non-amenable) and it's Connes' $T$ invariant is given by $$T(\Ucal_{\omega,\mu}(A,B)'')=\{t\,:\,\sigma_t^{\Omega}=\id=\sigma_t^{\widetilde{\omega}_K}\}=\bigcap_{\kappa=1}^K\{t\in\R\,:\,\sigma_t^{\widetilde{\omega}_\kappa}=\id\}.$$
Lemma \ref{LemmaModular} implies that, for all $1\leq\kappa\leq K$, $$\{t\in\R\,:\,\sigma_t^{\widetilde{\omega}_\kappa}=\id\}=\{t\in\R\,:\,\sigma_t^{\omega''}=\id\}\cap{\rm Sd}(Q_\kappa)^\perp.$$
Since $\bigcap_{\kappa=1}^K{\rm Sd}(Q_\kappa)^\perp={\rm Sd}(\mu)^\perp$, it concludes the computation of the $T$ invariant. By induction using \cite[Proposition 3.1]{Ue11b}, the $\tau$-invariant of $\U_{\omega,\mu}(A,B)''$ is the smallest topology on $\R$ making the maps $(t\mapsto\sigma_t^{\widetilde{\omega}_\kappa})$ continuous, for all $1\leq\kappa\leq K$. Let us denote by $\tau$ the $\tau$ invariant of $\U_{\omega,\mu}(A,B)''$. It follows from Lemma \ref{LemmaModular} that $\tau$ is the smallest topology on $\R$ containing $\tau(Q_\kappa)$ for all $1\leq\kappa\leq K$ and such that $(t\mapsto\sigma_t^{\omega''})$ is $\tau$-continuous. Since the topology generated by $\bigcup_{\kappa=1}^K\tau(Q_\kappa)$ is exactly $\tau(\mu)$, it concludes the proof of $(1)$.

\vspace{0.2cm}

\noindent$(2)$. By Theorem \ref{VNMN} we may and will assume that $K\geq 2$. Using the hypothesis, we can write $\Ucal_{\omega,\mu}(A,B)''=M_1*M_2$ with ${\rm dim}(M_1)\geq 2$ and ${\rm dim}(M_2)\geq 3$ and apply \cite[Theorem A]{BHR14}.

\vspace{0.2cm}

\noindent $(3)$ follows from from \cite[Corollary C]{CKSVW21} and $(4)$ follows from \cite{BDJ08}.\end{proof}

\section{Generalized Universal Quantum Homomorphism}\label{Section:general case}

\noindent Our goal in this section is to construct the universal quantum homomorphism for any pair $(A,B)$ of unital C*-algebra. In the first subsection, we show that, for non trivial $A$, the pair $(A,B)$ is matched if and only if $B$ is finite dimensional. This motivates us to deal with locally C*-algebra instead of classical C*-algebras. This is why the second subsection contains preliminaries on locally C*-algebras. In the third subsection we construct the universal quantum homomorphism $\rho\,:\, A\rightarrow B\ot \Ucal(A,B)$ with values in a locally C*-algebra $\Ucal(A,B)$. Finally the fourth subsection contains a study of the natural quantum semigroup structure on $\Ucal(A,A)$.

\subsection{Non-matched pairs}

\noindent Given a Banach space $X$ and $r>0$, we write $B_r(X):=\{x\in X\,:\,\Vert x\Vert\leq r\}$. The following simple Lemma is a direct consequence of Banach-Alaoglu's and Goldstine's Theorem.

\begin{lemma}\label{LemmaDual}
The natural map $B_1(X^*)\rightarrow B_1(X^{***})$ is weak*-weak*-continuous.
\end{lemma}

\begin{proof}
Let $(\omega_n)_n$ be a net in $B_1(X^*)$ converging weak* to $\omega\in B_1(X^*)$. It suffices to show that $\omega_n(\mu)\rightarrow\omega(\mu)$ for all $\mu\in B_1(X^{**})$. Let $\mathcal{Z}:=\{\mu\in B_1(X^{**})\,:\,\omega_n(\mu)\rightarrow\omega(\mu)\}$ and note that, since $\omega_n\rightarrow\omega$ weak* in $B_1(X^*)$, $\mathcal{Z}$ contains $B_1(X)$ (viewed as a subset of $B_1(X^{**})$). By weak*-density of $B_1(X)$ in $B_1(X^{**})$, it suffices to show that $\mathcal{Z}$ is weak*-closed in $B_1(X^{**})$. Let $\mu_k$ be a net in $\mathcal{Z}$ converging weak* to $\mu\in B_1(X^{**})$. Fix $\epsilon>0$ and note that $$B_2(X^*)=\bigcup_k\{\chi\in B_2(X^*)\,:\,\vert\mu_k(\chi)-\mu(\chi)\vert<\epsilon/2\}.$$
By weak*-compactness of $B_2(X^*)$, there exists a finite set $K$ such that $$B_2(X^*)=\bigcup_{k\in K}\{\chi\in B_2(X^*)\,:\,\vert\mu_{k}(\chi)-\mu(\chi)\vert<\epsilon/2\}.$$
Since $\mu_k\in\mathcal{Z}$ and $K$ is finite, there exists $N$ such that, for all $n\geq N$ one has $\vert\omega_n(\mu_k)-\omega(\mu_k)\vert<\epsilon/2$ for all $k\in K$. Fix $n\in\N$, since $\chi_n:=\omega-\omega_n\in B_2(X^*)$ there exists $k_n\in K$ such that $\vert\mu_{k_n}(\chi_n)-\mu(\chi_n)\vert<\epsilon/2$. Then, for all $n\geq N$,
\begin{eqnarray*}
\vert\omega_n(\mu)-\omega(\mu)\vert&\leq&
\vert\omega_n(\mu)-\omega_n(\mu_{k_n})+\omega(\mu_{k_n})-\omega(\mu)\vert+\vert\omega_n(\mu_{k_n})-\omega(\mu_{k_n})\vert\\
&<&\epsilon/2+\vert\mu_{k_n}(\chi_n)-\mu(\chi_n)\vert<\epsilon.
\end{eqnarray*}
Hence, $\mu\in \mathcal{Z}$.
\end{proof}

\noindent We will also use the following Lemma.

\begin{lemma}\label{LemmaSpecificElement}
Let $A$ be a (separable) unital C*-algebra. The following are equivalent.
\begin{enumerate}
    \item ${\rm dim}(A)\geq 2$.
\item There exists a unital $*$-homomorphism $\pi_0\,:\, A\rightarrow\Lcal(H)$, an element $a_0\in A$ and a norm continuous map $[0,1]\rightarrow\mathcal{U}(H)$, $t\mapsto u_t$, such that $u_0=1$ and $u_1\pi_0(a_0)u_1^*-\pi_0(a_0)+1$ is not invertible.
\end{enumerate}
\end{lemma}

\begin{proof}It suffices to prove $(1)\Rightarrow (2)$.
\vspace{0.2cm}

\noindent\textbf{Case 1. $A$ has no non-trivial projections and ${\rm dim}(A)\geq 2$.} In that case, $A$ must be infinite dimensional. By \cite[4.6.12]{KRbook}, there a self-adjoint element $a\in A$ such that the C*-algebra $C^*(a)\subset A$ generated by $a$ is infinite dimensional. Since $C^*(a)$ has no non-trivial projection either, ${\rm Sp}(a)\subset\R$ is connected. Hence, replacing $a$ by $\psi(a)$ if necessary, where $\psi\,:\,{\rm Sp}(a)\rightarrow[-1,1]$ is an homeomorphism, we may and will assume that ${\rm Sp}(a)=[-1,1]$. Let $\omega\in A^*$ be a state obtained by Hahn Banach extension of the state on $C^*(a)=C([-1,1])$ defined by integration with respect to the Lebesgue measure $\lambda$ on $[-1,1]$ and write $(H,\pi_0,\xi)$ the GNS-construction of $\omega$. Define $K:=\overline{\pi_0(C^*(a))\xi}\subset H$, let $p\in\Lcal(H)$ be the orthogonal projection onto $K$ and note that $p\in\pi_0(C^*(a))'$. Consider the unique self-adjoint unitary on $K$ such that $w\pi_0(f(a))\xi=\pi_0(f\circ\psi(a))\xi$, where $\psi\in{\rm Homeo}([-1,1])$ is the $\lambda$-preserving homeomorphism defined by $\psi(t)=-t$, for all $t\in\R$. Consider the self-adjoint unitary on $H$ defined by $u:=wp+(1-p)$. Then, a direct computation shows that $u\pi_0(a)u^*p=-\pi_0(a)p$ and $u\pi_0(a)u^*(1-p)=\pi_0(a)(1-p)$. It follows that $u\pi_0(a)u^*-\pi_0(a)=-2\pi_0(a)p$. In particular, the element $u\pi_0(a)u^*-\pi_0(a)$ is non-zero and self-adjoint so there exist a non-zero $\mu\in{\rm Sp}(u\pi_0(a)u^*-\pi_0(a))$. Hence, defining $a_0:=-\mu^{-1}a$, one has $u\pi_0(a_0)u^*-\pi_0(a_0)+1$ is not invertible. Since $u$ is a self-adjoint unitary and $u\notin\C1$, one has ${\rm Sp}(u)=\{-1,1\}$. Consider the map $h\in C({\rm Sp}(u))$ defined by: $$h(s)=\left\{\begin{array}{lcl}\pi&\text{if}&s=-1,\\0&\text{if}&s=1.\end{array}\right.$$
Since $s=e^{ih(s)}$ for all $s\in{\rm Sp}(u)$, we have $u=e^{ih(u)}$. Note that the map $t\mapsto u_t:=e^{ith(u)}$ is norm continuous, $u_0=1$ and $u_1=u$. It concludes the proof in case $1$.
\vspace{0.2cm}

\noindent\textbf{Case 2. $A$ has a non-trivial projection $p$.} Take any faithful representation $A\subset\Lcal(K)$ on a separable Hilbert space $K$ (here we use that $A$ is separable). Define $\mathcal{K}:=K\ot l^2(\N)$ and note that both $p\ot 1$ and $1\ot 1-p\ot1=(1-p)\ot 1$ are infinite projections in $\Lcal(\Kcal)$. Hence there exists a projection $q\in\Lcal(\Kcal)$ such that $q\leq p\ot 1$, $q\neq p\ot 1$ and $q\sim p\ot 1$. In particular, $1\ot 1-q\geq 1\ot 1-p\ot 1$ so $1\ot 1-q$ is infinite and, since $\Kcal$ is separable, $1\ot 1-q\sim 1\ot 1-p\ot 1$. It follows that there exists a unitary $u\in\Lcal(\Kcal)$ such that $q=u(p\ot 1)u^*$. Let $H:=\Kcal\ot\C^2$ and: $$\pi_0\,:\, A\rightarrow \Lcal(H)=M_2(\Lcal(K\ot l^2(\N))),\quad \pi_0(a)=\left(\begin{array}{cc}a\ot 1&0\\0&a\ot 1\end{array}\right).$$
Let $w_t:=\left(\begin{array}{cc}\cos(t\pi/2)&-\sin(t\pi/2)\\ \sin(t\pi/2)&\cos(t\pi/2)\end{array}\right)\in M_2(\C)$ and consider the norm continuous path of unitaries $t\mapsto u_t\,:\,[0,1]\rightarrow\mathcal{U}(H)$ defined by $u_t:=\left(\begin{array}{cc}u^*&0\\0&1\end{array}\right)w_t\left(\begin{array}{cc}u&0\\0&1\end{array}\right)w_t^*$. One has $u_0=\id_H$ and $u_1=\left(\begin{array}{cc}u^*&0\\0&u\end{array}\right)$. Hence,
$$u_1\pi_0(p)u_1^*-\pi_0(p)+\id_H=\left(\begin{array}{cc}u^*(p\ot 1)u-p\ot 1+1\ot 1&0\\0&q-p\ot 1+1\ot 1\end{array}\right).$$
Since $q\leq p\ot 1$ and $q\neq p\ot 1$, there exists a non-zero $\xi\in\Kcal$ such that $q\xi=0$ and $(p\ot 1)\xi=\xi$. Then,
$(u_1\pi_0(p)u_1^*-\pi_0(p)+\id_H)\left(\begin{array}{c}0\\ \xi\end{array}\right)=\left(\begin{array}{c}0\\ 0\end{array}\right)$. It follows that $u_1\pi_0(p)u_1^*-\pi_0(p)+\id_H$ is not invertible.
\end{proof}

\begin{theorem}\label{TheoremMatched}
If $A$ is a unital C*-algebra with ${\rm dim}(A)\geq 2$ then, for any unital C*-algebra $B$, the pair $(A,B)$ is matched if and only if $B$ is finite dimensional.
\end{theorem}

\begin{proof}
Suppose that $B$ is infinite dimensional and $(A,B)$ is matched. Let $\rho\,:\, A\rightarrow B\ot\U$ be the universal quantum homomorphism from $A$ to $B$. Let $\epsilon >0$ be small enough so that any element of  $O:=\{x\in\U\,:\,\Vert x-1\Vert<\epsilon\}$ is invertible.

\vspace{0.2cm}

\noindent By \cite[4.6.12]{KRbook}, there is an infinite dimensional abelian unital C*-subalgebra $D\subset B$. View $N:=D^{**}\subset M:=B^{**}$. Since $D$ is abelian, it has the Lance's weak expectation property i.e. there is a normal conditional expectation $E\,:\,M\rightarrow N$. In particular, for any $\omega\in D^*=N_*$, one has $\omega\circ E\in M_*=B^*$ and the map $\Psi\,:\,D^*\rightarrow B^*$, $\omega\mapsto \omega\circ E$ is linear and contractive. 

\vspace{0.2cm}

\noindent\textbf{Claim.} \textit{For all C*-algebra $C$ and all $Z\in B\ot C$, the map $B_1(D^*)\rightarrow C$, $(\omega\mapsto (\Psi(\omega)\ot\id)(Z))$ is continuous, where the unit ball $B_1(D^*)$ of $D^*$ has the weak*-topology and $C$ has the norm topology.}

\vspace{0.2cm}

\noindent\textit{Proof.} Define $\mathcal{Z}:=\{Z\in B\ot C\,:\,(\omega\mapsto (\Psi(\omega)\ot\id)(Z))\text{ is continuous from }B_1(D^*)\text{ to }C\}$,
where $D^*$ has the weak*-topology and $C$ has the norm topology. Take $Z=b\ot c$, $b\in B$ and $c\in C$ then, $(\Psi(\omega)\ot\id)(Z)=\Psi(\omega)(b)c=\omega(E(b))c$. By Lemma \ref{LemmaDual}, the map $\omega\mapsto \omega(E(b))$ is continuous and it follows that $b\ot c\in\mathcal{Z}$. Since $\mathcal{Z}$ is clearly a vector subspace of $B\ot C$, we deduce that the algebraic tensor product satisfies $B\odot C\subset\mathcal{Z}$. Hence it suffices to show that $\mathcal{Z}$ is norm closed in $B\ot C$. Let $Z_n\in B\ot C$ a sequence converging to $Z\in B\ot C$ and assume that $Z_n\in\mathcal{Z}$ for all $n$. Then, for all $\omega\in B_1(D^*)$,
$$\Vert(\Psi(\omega)\ot\id)(Z_n)-(\Psi(\omega)\ot\id)(Z)\Vert=\Vert (\Psi(\omega)\ot\id)(Z_n-Z)\Vert\leq\Vert Z_n-Z\Vert\rightarrow 0.$$
Hence, the map $(\omega\mapsto (\Psi(\omega)\ot\id)(Z))$ being the uniform limit of the continuous maps $(\omega\mapsto (\Psi(\omega)\ot\id)(Z_n))$ is itself continuous.\hfill\qed

\vspace{0.2cm}

\noindent\textit{End of the proof of Theorem \ref{TheoremMatched}.} By Lemma \ref{LemmaSpecificElement}, there a unital $*$-homomorphism $\pi_0\,:\, A\rightarrow\Lcal(H)$, an element $a_0\in A$ and a norm continuous map $[0,1]\rightarrow\mathcal{U}(H)$, $t\mapsto u_t$, such that $u_0=1$ and $u_1\pi_0(a_0)u_1^*-\pi_0(a_0)+1$ is not invertible. Consider the unitary representation $\pi_t\,:\,A\rightarrow\mathcal{L}(H)$, $\pi_t(a)=u_t\pi_0(a)u_t^*$. Note that, for all $a\in \Lcal(H)$, the map $[0,1]\rightarrow \Lcal(H)$, $t\mapsto\pi_t(a)$ is norm continuous.

\vspace{0.2cm}

\noindent Since $D$ is infinite dimensional, its spectrum ${\rm Sp}(D)$ is infinite and since it is also compact (for the topology induced by the weak* topology on $B_1(D^*)$), ${\rm Sp}(D)$ is not discrete. Hence, there exists $\chi_0\in{\rm Sp}(D)$ and a sequence $(\chi_n)_n$ converging to $\chi_0$ such that $\chi_n\neq\chi_0$ for all $n$. By the Claim, the map ${\rm Sp}(D)\rightarrow \U$, $\chi\mapsto(\Psi(\chi)\ot\id)\rho(a_0)$ is continuous so $(\Psi(\chi_n)\ot\id)\rho(a_0)\rightarrow(\Psi(\chi_0)\ot\id)\rho(a_0)$ in norm in $\U$. In particular, there exists $\chi_1\neq\chi_0$ such that
$\Vert (\Psi(\chi_1)\ot\id)\rho(a_0)-(\Psi(\chi_0)\ot\id)\rho(a_0)\Vert <\epsilon$. Define $x:=(\Psi(\chi_1)\ot\id)\rho(a_0)-(\Psi(\chi_0)\ot\id)\rho(a_0)+1\in\U$. Then,
$\Vert x-1\Vert=\Vert (\Psi(\chi_1)\ot\id)\rho(a_0)-(\Psi(\chi_0)\ot\id)\rho(a_0)\Vert <\epsilon,$
so $x\in O$. By Urysohn's Lemma, there is a continuous function $f\,:\,{\rm Sp}(D)\rightarrow[0,1]$ such that $f(\chi_0)=0$ and $f(\chi_1)=1$. Note that the map ${\rm Sp}(D)\rightarrow \Lcal(H)$, $\chi\mapsto\pi_{f(\chi)}(a)$ is continuous, for all $a\in A$ so it defines an element $\pi'(a)\in C({\rm Sp}(D))\ot\Lcal(H)=D\ot\Lcal(H)\subset B\ot\Lcal(H)$ and $\pi'\,:\,A\rightarrow B\ot\Lcal(H)$ is clearly a unital $*$-homomorphism. Let $\widetilde{\pi}\,:\,\U\rightarrow\Lcal(H)$ be the unique unital $*$-homomorphism such that $(\id\ot\widetilde{\pi})\rho=\pi'$. Since, for all $\chi\in{\rm Sp}(D)$ and all $d\in D$, $\Psi(\chi)(d)=\chi(d)$ we have,
\begin{eqnarray*}
\widetilde{\pi}(x)&=&(\Psi(\chi_1)\ot\id)\pi'(a_0)-(\Psi(\chi_0)\ot\id)\pi'(a_0)+1=\pi_{f(\chi_1)}(a_0)-\pi_{f(\chi_0)}(a_0)+1\\
&=&\pi_1(a_0)-\pi_0(a_0)+1=u_1\pi_0(a_0)u_1^*-\pi_0(a_0)+1.
\end{eqnarray*}
It follows that $\widetilde{\pi}(x)$ is not invertible in $\Lcal(H)$ so $x$ is not invertible in $\U$, a contradiction.\end{proof}

\noindent The previous Theorem shows that the only way to construct a universal quantum homomorphism from $A$ to $B$ for an infinite dimensional $B$ is to extend the notion of quantum homomorphism to a larger category that the category of unital C*-algebras. As we will see, the natural category to do that is the category of unital locally C*-algebras.

\vspace{0.2cm}

\subsection{Preliminaries on projective limits and locally C*-algebras}\label{SectionLocally}
In this section, we recall the construction of projective limit of C*-algebras as well as the definition and basic properties of locally C*-algebras. For the convenience of the reader, we keep this section as self contained as possible by including proofs of known results that we will use later.  Recall that a C*-seminorm on a unital $*$-algebra $A$ is a seminorm $p$ on the vector space $A$ also satisfying:
$$p(ab)\leq p(a)p(b),\quad p(a^*)=p(a),\quad p(1)=1,\quad p(a^*a)=p(a)^2.$$
For more background on locally C*-algebras, we refer the interested reader to \cite{Fr05}.
\subsubsection{Projective limits} Let $(S,\leq)$ be a directed poset, for each $s\in S$, $A_s$ a unital C*-algebra and for each $s\leq t$, $\pi_{s,t}\,:\, A_t\rightarrow A_s$ a unital $*$-homomorphism such that $\pi_{s,s}=\id_{A_s}$ and for all $r\leq s\leq t$ one has $\pi_{r,s}\circ\pi_{s,t}=\pi_{r,t}$. Define the unital $*$-algebra $$A:=\{(a_s)_s\in\prod_{s\in S} A_s\,:\,\pi_{s,t}(a_t)=a_s\text{ for all }s\leq t\}.$$
Note that we have, for all $s\in S$, the canonical projection is a unital $*$-homomorphism $\pi_{s}\,:\, A\rightarrow A_s$ satisfying $\pi_{s,t}\circ\pi_{t}=\pi_{s}$ for all $s\leq t$ and the maps $\pi_{s}$ separate the points i.e. for all $x\in A$ one has $x\neq 0$ $\Rightarrow$ $\exists s\in S$ such that $\pi_{s}(x)\neq 0$. Consider the smallest topology on $A$ for which all the maps $\pi_{s}$ are continuous (this is the restriction of the product topology). Note that the topology on $A$ is the same as the topology defined by the family of C*-seminorms $x\mapsto\Vert\pi_{s}(x)\Vert$ on $A$. Then $A$ becomes a complete Hausdorff topological unital $*$-algebra whose topology is generated by a family of C*-semi-norms (in other words, a locally C*-algebra, see Section \ref{SectionLocallyC*algebra}), called the projective limit of the projective system of unital C*-algebras with unital $*$-homomorphisms $(A_s,\pi_{s,t},S)$ and it is also denoted by $A=\underset{\leftarrow}{\lim}\,A_s$.

\begin{remark}\label{RmkProj} Let $A=\underset{\leftarrow}{\lim}\,A_s$. The following holds.
\begin{enumerate}
\item For any unital C*-algebra $D$ with unital $*$-homomorphisms $\nu_s\,:\, D\rightarrow A_s$ such that $\pi_{s,t}\circ\nu_t=\nu_s$ for all $s\leq t$ there exists a unique unital $*$-homomorphism $\widetilde{\nu}\,:\,D\rightarrow A$ such that $\pi_{s}\circ\widetilde{\nu}=\nu_s$ for all $s\in S$. Indeed, it suffices to define $\widetilde{\nu}(d)=(\nu_s(d))_{s\in S}$. In fact this property is true for any unital $*$-algebra $D$ so in particular, for any unital locally C*-algebra $D$ and, if $D$ is a unital locally C*-algebra and we assume the morphisms $\nu_s$ to be continuous then the morphism $\widetilde{\nu}$ is continuous, by definition of the topology on $A$.
\item For any unital C*-algebra $D$, any unital $*$-homomorphism $\rho\,:\, D\rightarrow A$ is continuous. Indeed, for all $s\in S$, $\pi_{s}\circ\rho\,:\,D\rightarrow A_s$ is a unital $*$-homomorphism between C*-algebras hence continuous.
\end{enumerate}
\end{remark}

\begin{lemma}\label{LemmaReps}
Let $A=\underset{\leftarrow}{\lim} A_s$. For any continuous unital $*$-homomorphism $\pi\,:\,A\rightarrow\Lcal(H)$ there exists $s\in S$ and a unital $*$-homomorphism $\widetilde{\pi}\,:\,\overline{\pi_s(A)}\rightarrow \Lcal(H)$ such that $\pi=\widetilde{\pi}\circ\pi_s$.
\end{lemma}

\begin{proof}
Since the topology on $A$ is generated by the family of seminorms $(a\mapsto\Vert\pi_s(a)\Vert)$, $s\in S$, it follows that there exists a finite subset $F\subset S$ and a $C>0$ such that $\Vert\pi(a)\Vert\leq C\underset{t\in F}{\max}\Vert\pi_t(a)\Vert$ for all $a\in A$. Now, since $S$ is directed, there exists $s\in S$ such that, for all $a\in A$, $\underset{t\in F}{\max}\Vert\pi_t(a)\Vert\leq\Vert\pi_s(a)\Vert$. Hence, $\ker(\pi_s)\subset\ker(\pi)$ so that the map $\pi_s(A)\rightarrow\pi(A)\subset\Lcal(H)$, $\pi_s(a)\mapsto\pi(a)$ is well defined. It is easy to check that this map is a unital $*$-homomorphism. It is continuous so it extends to a unital $*$-homomorphism $\overline{\pi_s(A)}\rightarrow\overline{\pi(A)}\subset\Lcal(H)$.
\end{proof}

\begin{lemma}\label{LemmaDensity} A subset $E\subseteq A$ is dense if and only if $\pi_{s}(E)$ is dense in $\pi_{s}(A)$, for all $s\in S$.
\end{lemma}
\begin{proof} Let $x_0\in A$ be any element. Assuming that $\pi_{s}(E)\subseteq \pi_{s}(A)$ is dense, we can find, for any $s\in S$ and $n\geq 1$, an element $x_{s,n}\in E$ such that $\Vert\pi_{s}(x_{s,n}-x_0)\Vert_{A_s}<\frac{1}{n}$. Consider the net $(x_{s,n})_{s,n}$ in $E$. Fix $r\in S$ and let $\epsilon>0$. Let $N\in\N$ with $N>\frac{1}{\epsilon}$. For all $n\geq N$ and $s\geq r$:
\begin{eqnarray*}
\Vert\pi_r(x_{s,n}-x_0)\Vert&=&\Vert\pi_{r,s}\pi_s(x_{s,n}-x_0)\Vert\leq\Vert\pi_s(x_{s,n}-x_0)\Vert<\epsilon.
\end{eqnarray*}
Hence, the net $(x_{s,n})_{s,n}$ in $E$ converges to $x_0$.
\end{proof}

\subsubsection{Locally C*-algebras}\label{SectionLocallyC*algebra} A \textit{unital locally C*-algebra} is a complete Hausdorff topological unital $*$-algebra whose topology is generated by a family of C*-semi-norms. Given a unital locally C*-algebra $A$ we denote by $S(A)$ the directed set of all continuous C*-seminorms on $A$ equipped with the obvious order $p\leq q\Leftrightarrow p(a)\leq q(a)$ $\forall a\in A$. It is easy then to check that for a net $(x_i)_{i\in I}$ in $A$, we have $x_i\rightarrow x$ if and only if $r(x_i)\rightarrow r(x)$ for all $r\in S(A)$. We refer to \cite{Fr05} for more background on locally C*-algebras. As noted above, a projective limit of the form $A=\underset{\leftarrow}{\lim} A_s$ gives an example of a locally C*-algebra.

\vspace{0.2cm}

\noindent Let $p\in S(A)$ and note that $\ker(p)$ is a closed $*$-ideal so that $A_p:=A/\ker(p)$ becomes a unital $*$-algebra with a C*-norm. Let $\pi_p\,:\, A\rightarrow A_p$ the canonical quotient map, it is a unital $*$-homomorphism which is continuous since $p(a)=\Vert\pi_p(a)\Vert$ for all $a\in A$. By definition, for each $p,q\in S(A)$ if $p\leq q$ then there is a unique unital $*$-homomorphism $\pi_{p,q}\,:\, A_q\rightarrow A_p$ such that $\pi_{p,q}\circ\pi_q=\pi_p$. Note that for all $p\in S(A)$, $\pi_{p,p}=\id_{A_p}$ and for all $p,q,r\in S(A)$ if $p\leq q\leq r$ then $\pi_{p,q}\circ\pi_{q,r}=\pi_{p,r}$. Moreover, the maps $\pi_{p,q}$ are surjective.

\begin{lemma}\label{LemmaClosed}
The normed unital $*$-algebra $A_p$ is complete and so it is a C*-algebra.
\end{lemma}

\begin{proof}
View $A_p\subset\overline{A_p}$ its completion. Note that, for $p\leq q$, the unital $*$-homomorphism $\pi_{p,q}\,:\, A_q\rightarrow A_p\subset \overline{A_p}$ is a contraction and so it extends uniquely to a unital $*$-homomorphism $\pi_{p,q}\,:\,\overline{A_q}\rightarrow\overline{A_p}$. We first prove the following Claim.

\vspace{0.2cm}

\noindent\textbf{Claim.} \textit{$\forall a\in A$ with $p(a)<1$ $\exists x\in A$ such that $\pi_p(x)=\pi_p(a)$ and $q(x)<1$ for all $q\in S(A)$.}

\vspace{0.2cm}

\noindent\textit{Proof of the Claim.} Let $t:=\max\{1/2,p(a)^2\}<1$ and define: $$f\,:\,\R\rightarrow\R_+\quad f(s):=\left\{\begin{array}{lcl}\frac{1+s}{1-t}&\text{if}&-1\leq s\leq -t,\\ \frac{1-s}{1-t}&\text{if}&t\leq s\leq 1,\\
1&\text{if}&-t\leq s\leq t,\\
0&\text{if}&s\leq -1\text{ or }s\geq 1.\end{array}\right.$$
Then $f\in C(\R)$ and since $\Vert\pi_p(a^*a)\Vert=p(a)^2\leq t$, one has $f(\pi_p(a^*a))=1$ in the C*-algebra $\overline{A_p}$. Define, for all $q\in S(A)$ the element $y_q:=\pi_q(a)f(\pi_q(a^*a))\in \overline{A_q}$. By density of $A_q$ in $\overline{A_q}$ and since $\pi_q$ is surjective, there exists for each $q\in S(A)$ and all $n\in\N$, an element $x_{q,n}\in A$ such that $\Vert\pi_q(x_{q,n})-y_q\Vert<1/n$. Let us show that the net $(x_{q,n})_{q\in S(A)}$ is a Cauchy net. Fix $r\in S(A)$ and $\epsilon>0$ and note that, for all $(q,n),(q',n')\in S(A)\times\N$ with $(q,n),(q',n')\geq (r,N)$, where $N\in\N$, $N>2/\epsilon$, one has:
\begin{eqnarray*}
r(x_{q,n}-x_{q',n'})&=&\Vert\pi_r(x_{q,n}-x_{q',n'})\Vert=\Vert\pi_{r,q}\pi_q(x_{q,n})-\pi_{r,q'}\pi_{q'}(x_{q',n'})\Vert\\
&\leq&\Vert\pi_{r,q}\pi_q(x_{q,n})-\pi_{r,q}(y_q)\Vert+\Vert\pi_{r,q}(y_q)-\pi_{r,q'}(y_{q'})\Vert\\
&&+\Vert\pi_{r,q'}(y_{q'})-\pi_{r,q'}\pi_{q'}(x_{q',n'})\Vert\\
&<&\epsilon+\Vert\pi_{r,q}(y_q)-\pi_{r,q'}(y_{q'})\Vert=\epsilon
\end{eqnarray*}
since $\pi_{r,q}(y_q)=\pi_{r,q}(\pi_q(a))\pi_{r,q}(f(\pi_q(a^*a)))=\pi_r(a)f(\pi_{r,q}\pi_q(a^*a))=\pi_r(a)f(\pi_r(a^*a))$ for all $q\in S(A)$ with $q\geq r$. Since $A$ is complete, there exists $x\in A$ such that $\lim x_{q,n}=x$. Note that, for any $q\geq r$, $r(x_{q,n})=\Vert\pi_r(x_{q,n})\Vert=\Vert\pi_{r,q}\pi_q(x_{q,n})\Vert\leq\Vert\pi_q(x_{q,n})\Vert
<1/n+\Vert y_q\Vert$. Moreover,
$\Vert y_q\Vert^2=\Vert y_q^*y_q\Vert=\Vert g(\pi_q(a^*a)\Vert$, where $g\in C(\R)$, $g(s)=sf(s)^2$. A direct computation gives $\underset{s\in\R}{\sup}\vert g(s)\vert=t$ hence, $\Vert y_q\Vert\leq\sqrt{t}$. It follows that, for all $r\in S(A)$, $$r(x)=\lim_{q,n} r(x_{q,n})\leq\sqrt{t}<1.$$
Note that $y_p=\pi_p(a)f(\pi_p(a^*a))=\pi_p(a)$ since $f(\pi_p(a^*a))=1$. Hence, for all $q\geq p$,
$$\pi_{p,q}(y_q)=\pi_{p,q}\pi_q(a)\pi_{p,q}(f(\pi_q(a^*a)))=\pi_p(a)f(\pi_{p,q}\pi_q(a^*a))=\pi_p(a)f(\pi_p(a^*a))=\pi_p(a).$$
It follows that, $\forall q\geq p$, $\Vert\pi_p(x_{q,n})-\pi_p(a)\Vert=\Vert\pi_{p,q}\pi_q(x_{q,n})-\pi_{p,q}(y_q)\Vert\leq\Vert\pi_q(x_{q,n})-y_q\Vert\leq 1/n$. We conclude that $p(x-a)=\lim_{q,n}\Vert\pi_p(x_{q,n})-\pi_p(a)\Vert=0$. Hence, $\pi_p(x)=\pi_p(a)$.

\vspace{0.2cm}

\noindent\textit{End of the proof of the Lemma.} Let $y\in\overline{\pi_p(A)}$ and take a sequence $(y_n)_{n}$ in $\pi_p(A)$ converging to $y$. Dropping to a subsequence if necessary, we have that $\Vert y_{n+1}-y_{n}\Vert<2^{-n}$ for all $n\geq 1$. By the Claim, for all $n\geq 1$, there exists $x_n\in A$ such that $q(x_n)<2^{-n}$ for all $q\in S(A)$ and $\pi_p(x_n)=y_{n+1}-y_{n}$. It follows that $\left(\sum_{k=1}^nx_k\right)_{n\geq 1}$ is a Cauchy sequence in $A$. Let $\sum_{k=1}^{+\infty} x_k\in A$ be its limit and define $x:=x_0+\sum_{k=1}^{+\infty} x_k$, where $x_0\in A$ is such that $\pi_p(x_0)=y_1$. By continuity we have, $\pi_p(x)=y_1+\sum_{k=1}^{+\infty}(y_{k+1}-y_{k})=y$ and hence, we are done.
\end{proof}

\begin{remark}
$S(A)$ is exactly the set of seminorms of the form $(a\mapsto\Vert\pi(a)\Vert)$ for $\pi\,:\, A\rightarrow\Lcal(H)$ a continuous unital $*$-homomorphism. Indeed, any such map is clearly a continuous C*-norm and, if $p\in S(A)$, $A_p:=A/\ker(p)$ is a C*-algebra so can be embedded into $\Lcal(H)$ and the canonical continuous map $\pi_p\,:\,A\rightarrow A_p\subset\Lcal(H)$ satisfies $p(a)=\Vert\pi_p(a)\Vert$ for all $a\in A$.
\end{remark}

\begin{corollary}\label{CorClosed}
Let $A$ be a unital locally C*-algebra and $\pi\,:\,A\rightarrow \Lcal(H)$ be a continuous unital $*$-homomorphism. Then, $\pi(A)$ is closed so it is a C*-algebra.
\end{corollary}

\begin{proof}
Let $p\in S(A)$ defined by $p(a)=\Vert\pi(a)\Vert$. Then $\ker(p)=\ker(\pi)$ so that $\pi$ can be identified with $\pi_p$ and $\pi(A)$ with $A_p$ which is closed, by Lemma \ref{LemmaClosed}.
\end{proof}

\noindent Consider the projective system $(A_p,\pi_{p,q},S(A))$ of unital C*-algebras, associated to the locally C*-algebra $A$. Then, the following shows the equivalence of $A$ and the projective limit $\underset{\leftarrow}{\lim}\,A_p$ as locally C*-algebras, which is referred to as the \textit{Arens-Michael decomposition} of $A$.

\begin{proposition}\label{PropAM}
The map: $$\psi\,:\,A\rightarrow\underset{\leftarrow}{\lim}\,A_p,\quad a\mapsto (\pi_p(a))_{p\in S(A)}$$ is an isomorphism of unital locally C*-algebras.
\end{proposition}

\begin{proof} It is straightforward to check that $\psi$ is well defined and is a unital $*$-homomorphism. Since $A$ is Hausdorff, $S(A)$ separates the points which easily implies that $\psi$ is injective. Let $b=(b_p)_{p\in S(A)}\in\underset{\leftarrow}{\lim}A_p$ be any element. For each $p\in S(A)$ choose $a_p\in A$ such that $\pi_p(a_p)=b_p$. Then, the net $(a_p)_{p\in S(A)}$ is a Cauchy net. Indeed, for any $r\in S(A)$, if $p,q\geq r$ then,  $r(a_p-a_q)=\Vert\pi_r(a_p)-\pi_r(a_q)\Vert=\Vert\pi_{r,p}\pi_p(a_p)-\pi_{r,q}\pi_q(a_q)\Vert=\Vert\pi_{r,p}(b_p)-\pi_{r,q}(b_q)\Vert=0$. Since $A$ is complete, there exists $a\in A$ such that, for any $r\in S(A)$, $r(a-a_p)\rightarrow_p0$. Then, for any $r\in S(A)$ and any $\epsilon>0$ there exists $p\in S(A)$, $p\geq r$ such that $r(a-a_p)<\epsilon$. Hence, for all $r\in S(A)$ and all $\epsilon>0$
$$\epsilon>r(a-a_p)=\Vert\pi_r(a)-\pi_r(a_p)\Vert=\Vert\pi_{r}(a)-\pi_{r,p}\pi_p(a_p)\Vert=\Vert\pi_{r}(a)-\pi_{r,p}(b_p)\Vert=\Vert\pi_{r}(a)-b_r\Vert$$
It follows that $\pi_r(a)=b_r$ for all $r\in S(A)$, so $\psi$ is surjective. By definition, $\psi$ is continuous, therefore it suffices to show that it is open. Let $O\subset A$ be a non-empty open set and $x=(x_p)_p\in\psi(O)$. Let $a\in O$ such that $\psi(a)=(\pi_p(a))_p=(x_p)_p$. By definition of the topology on $A$, there is a finite set of continuous seminorms $F\subset S(A)$ and $\epsilon>0$ such that $$V:=\{b\in A\,:\,p(a-b)=\Vert x_p-\pi_p(b)\Vert<\epsilon\text{ for all }p\in F\}\subset O.$$
Consider the open set $W:=\{y=(y_p)_p\in \underset{\leftarrow}{\lim} A_p\,:\,\Vert x_p-y_p\Vert<\epsilon\text{ for all }p\in F\}\subset\underset{\leftarrow}{\lim} A_p$. Then $x\in W$ and $\psi(V)\subset W$. Since $\psi$ is surjective, it is easy to see that $W\subset \psi(V)$. But $\psi(V)\subset\psi(O)$ so $W\subset\psi(O)$ and $\psi(O)$ is indeed open.\end{proof}

\noindent Thus, as shown in the previous proposition, unital locally C*-algebras can always be obtained as projective limits of projective systems of unital C*-algebras. We will freely use this fact when needed.

\subsubsection{Projective limits of unital locally C*-algebras} Let $(S,\leq)$ be a directed poset, for each $s\in S$, $A_s$ a unital locally C*-algebra and for each $s\leq t$, $\pi_{s,t}\,:\, A_t\rightarrow A_s$ a unital continuous $*$-homomorphism such that $\pi_{s,s}=\id_{A_s}$ and for all $r\leq s\leq t$ one has $\pi_{r,s}\circ\pi_{s,t}=\pi_{r,t}$. Define the topological unital $*$-algebra $$A:=\{(a_s)_s\in\prod_{s\in S} A_s\,:\,\pi_{s,t}(a_t)=a_s\text{ for all }s\leq t\},$$
where the topology is given by the product topology i.e. the smallest topology for which the canonical projections $\pi_{s}\,:\, A\rightarrow A_s$ are continuous, for all $s\in S$. Note that $\pi_{s,t}\circ\pi_{t}=\pi_{s}$ for all $s\leq t$ and the maps $\pi_{s}$ separate the points i.e. for all $x\in A$ one has $x\neq 0$ $\Rightarrow$ $\exists s\in S$ such that $\pi_{s}(x)\neq 0$. Note that the topology on $A$ is the same as the topology defined by the family of C*-seminorms $x\mapsto p(\pi_{s}(x))$ on $A$ (for $p\in S(A_s)$ and $s\in S$) and it is easy to check that $A$ is indeed complete and Hausdorff. Thus, $A$ is a locally C*-algebra, called the projective limit of the projective system of unital locally C*-algebras with unital continuous $*$-homomorphisms $(A_s,\pi_{s,t},S)$ and it is also denoted by $A=\underset{\leftarrow}{\lim}\,A_s$.

\begin{remark}\label{Rmkreps}
Note that for projective limits of unital locally C*-algebras, a result similar to Lemma \ref{LemmaReps} holds. In other words, for a projective limit of unital locally C*-algebras $A=\underset{\leftarrow}{\lim}\,A_s$ and given any continuous unital $*$-homomorphism $\pi:A\rightarrow \Lcal(H)$, there exists $s\in S$ and a unital, continuous, $*$-homomorphism $\Tilde{\pi}:\overline{\pi_s(A)}\rightarrow \Lcal(H)$ such that $\pi=\widetilde{\pi}\circ\pi_s$. The proof is similar to the proof of Lemma \ref{LemmaReps}, whereby we first obtain finitely many $s_1,...,s_n$ in $S$ and corresponding semi-norms $p_1,...,p_n$ on $A_{s_1},...,A_{s_n}$, respectively, such that for any $a\in A$, $\|\pi(a)\|\leq C$max$_{i=1,...,n}(p_i(\pi_{s_i}(a)))$. Choosing $s\in S$ such that $s\geq s_i$ for all $i=1,...,n$, we have that the map $\pi_s(a)\mapsto \pi(a)$ is well defined as if $\pi_s(a)=0$ then $p_i(\pi_{s_i}(a))=0$ for all $i=1,...,n$ and thus, $\pi(a)=0$, which defines the map $\widetilde{\pi}$ on $\pi_s(A)$. Using the inequality $\|\pi(a)\|\leq C$max$_{i=1,...,n}(p_i(\pi_{s_i}(a)))$ and completeness of $\Lcal(H)$, $\widetilde{\pi}$ can be extended uniquely to $\overline{\pi_s(A)}$ and having the property that $\|\widetilde{\pi}(t)\|\leq C$max$_{i=1,...,n}(p_i(\pi_{s_i,s}(t)))$ for all $t$ in $\overline{\pi_s(A)}$. It is then easily checked $\widetilde{\pi}$ is a unital, continuous $*$-homomorphism and by construction, $\pi=\widetilde{\pi}\circ\pi_s$.
\end{remark}

\subsubsection{Minimal tensor product of unital locally C*-algebras} Let us now recall the minimal tensor product construction of unital locally C*-algebras. Let $A,B$ be two unital locally C*-algebras and consider the unital $*$-algebra given by the algebraic tensor product $A\odot B$ with the topology $\tau_{{\rm min}}$ generated by the family of C*-seminorms given by $$A\odot B\rightarrow \R_+,\quad x\mapsto\Vert (\pi\ot\rho)(x)\Vert_{\Lcal(H\ot K)},$$ for $\pi\,:\,A\rightarrow\Lcal(H)$ and $\rho\,:\,B\rightarrow\Lcal(K)$ continuous unital $*$-homomorphism. Then, its completion as a topological vector space is a unital locally C*-algebra denoted by $A\ot B$ and called the minimal tensor product of $A$ and $B$. Note that, by definition, $\tau_{\rm min}$ is the smallest topology of $A\odot B$ making the morphisms $\pi\ot\rho\,:\, A\odot B\rightarrow \Lcal(H\ot K)$ continuous, for all continuous unital $*$-homomorphism $\pi\,:\, A\rightarrow \Lcal(H)$ and $\rho\,:\, B\rightarrow \Lcal(K)$. Moreover, if $A$ and $B$ are C*-algebras we recover the minimal tensor product of C*-algebras.

\begin{proposition}\label{PropMinTens}
Let $A_1,A_2,B_1,B_2$ be any unital locally C*-algebras and $\rho_k\,:\, A_k\rightarrow B_k$ be a continuous unital $*$-homomorphism for $k=1,2$. There exists a unique continuous $*$-homomorphism $\rho_1\ot\rho_2\,:\, A_1\ot A_2\rightarrow B_1\ot B_2$ such that:
$$(\rho_1\ot\rho_2)(a_1\ot a_2)=\rho_1(a_1)\ot\rho_2(a_2)\text{ for all }a_k\in A_k, k=1,2.$$
\end{proposition}

\begin{proof}
Since $A_1\odot A_2$ is dense in $A_1\ot A_2$, $B_1\ot B_2$ is complete and $\rho_1\otimes\rho_2$ is well defined and linear on $A_1\odot A_2$, it suffices to show that $\rho_1\ot\rho_2\,:\, A_1\odot A_2\rightarrow B_1\otimes B_2$ is $\tau_{\rm min}$-continuous. To this end, let $\pi_k\,:\, B_k\rightarrow\Lcal(H_k)$, $k=1,2$ be continuous unital $*$-homomorphisms. Then $\pi_k\circ\rho_k\,:\,A_k\rightarrow\Lcal(H_k)$ is continuous, so $\pi_1\circ\rho_1\ot\pi_2\circ\rho_2=(\pi_1\ot\pi_2)(\rho_1\ot\rho_2)\,:\,A_1\odot A_2\rightarrow\Lcal(H_1\ot H_2)$ is $\tau_{\rm min}$-continuous. It follows that $\rho_1\ot\rho_2$ is continuous.\end{proof}

\begin{proposition}\label{CorTensor}
Let $A,B$ be two unital locally C*-algebras. For any $\omega\in A^*$, the mapping $x\mapsto (\omega\ot\id)(x)$ on $A\odot B$ to $B$ extends to a linear, continuous map from $A\otimes B$ to $B$ which satisfies the property that for any unital locally C*-algebra $C$ and any continuous unital $*$-homomorphism $\rho\,:\, B\rightarrow C$, $\rho((\omega\ot\id)(x))=(\omega\ot\id)(\id\ot\rho)(x)$. A similar statement holds for the mapping defined as $x\mapsto (\id\ot\omega)(x)$, where $\omega\in B^*$.
\end{proposition}

\begin{proof}
Since $\omega$ is linear and continuous between locally convex spaces, thus there exists $p\in S(A)$ such that $\vert\omega(a)\vert\leq p(a)$ for all $a\in A$, as in the proof of Lemma \ref{LemmaReps}. Then $\ker(p)\subset\ker(\omega)$ and there exists a unique linear $\widetilde{\omega}\,:\, A_p=A/\ker(p)\rightarrow\C$ such that $\widetilde{\omega}\circ\pi_p=\omega$. It is clear that $\widetilde{\omega}$ is continuous on the C*-algebra $A_p$. By definition, for each $q\in S(B)$ with canonical continuous morphism $\rho_q\,:\, B\rightarrow B_q=B/\ker(q)$, the unital $*$-homomorphism $\pi_p\ot\rho_q\,:\,A\odot B\rightarrow A_p\ot B_q$ is $\tau_{\rm min}$-continuous, hence, the linear map $\omega\ot\rho_q=(\widetilde{\omega}\ot\id)(\pi_p\ot\rho_q)\,:\, A\odot B\rightarrow B_q$ is $\tau_{\rm min}$-continuous, for all $q\in S(B)$. It follows that the well defined linear map $\omega\ot\id\,:\, A\odot B\rightarrow B$ extends to a continuous linear map $\omega\ot\id\,:\, A\ot B\rightarrow B$. It is easy to check that $(\omega\ot\id)(x)$, for $x\in A\odot B$, satisfies the property of the statement and thus, also for $x\in A\ot B$ by density and continuity.\end{proof}

\noindent Recall that $\omega\in A^*$ is called positive if $\omega(a^*a)\geq 0$ for all $a\in A$.

\begin{remark}
If $\omega\in A^*$ and $\mu\in B^*$ then there exists a unique $\omega\ot\mu\in (A\ot B)^*$ such that $(\omega\ot\mu)(a\ot b)=\omega(a)\mu(b)$ for all $a\in A$, $b\in B$. Moreover, if both $\omega$ and $\mu$ are positive then $\omega\ot\mu$ is positive. Indeed, we known that there exists $p\in S(A)$, $q\in S(B)$ , $\widetilde{\omega}\in A_p^*$ and $\widetilde{\mu}\in B_q^*$ such that $\omega=\widetilde{\omega}\pi_p $ and $\mu=\widetilde{\mu}\rho_q$. Then $\pi_p\ot\rho_q\,:\,A\ot B\rightarrow A_p\ot B_q$ is a continuous unital $*$-homomorphism and we can define $\omega\ot\mu:=(\widetilde{\omega}\ot\widetilde{\mu})(\pi_p\ot\rho_q)\in (A\ot B)^*$. The results now follow.
\end{remark}

\noindent Let now $A=\underset{\leftarrow}{\lim}{A_s}$ ($s\in S$) be a projective limit of unital locally C*-algebra, with connecting maps denoted $\pi_{s',s}\,:\, A_s\rightarrow A_{s'}$ for $s'\leq s$ and the projective maps denoted $\pi_s:A\rightarrow A_s$, for all $s\in S$. Similarly, let $B=\underset{\leftarrow}{\lim}\,B_t$ ($t\in T$) also be a projective limit of unital locally C*-algebra, with connecting maps $\rho_{t',t}\,:\,B_t\rightarrow B_{t'}$ for $t'\leq t$ and projective maps $\rho_t:B\rightarrow B_t$ for all $t\in T$. Consider the projective system of unital locally C*-algebras $(A_s\ot B_t)_{(s,t)\in S\times T}$ where the connecting maps are $$\psi_{(s',t'),(s,t)}:=\pi_{s',s}\ot\rho_{t',t}\,:\,A_s\ot B_t\rightarrow A_{s'}\ot B_{t'}$$ for $(s',t')\leq (s,t)$ (i.e. $s'\leq s$ and $t'\leq t$), defined as in Proposition \ref{PropMinTens} and with the canonical projective maps denoted $\psi_{(s_0,t_0)}\,:\,\underset{\leftarrow}{\lim} (A_s\ot B_t)\rightarrow A_{s_0}\ot B_{t_0}$, for any $s_0\in S, t_0\in T$.

\begin{proposition}\label{PropAMTensor} Let $A=\underset{\leftarrow}{\lim}{A_s}$ and $B=\underset{\leftarrow}{\lim}\,B_t$ be projective limits of unital locally C*-algebras as above. Assume that the projective maps $\pi_s$, $\rho_t$ are surjective for all $s\in S$ and $t\in T$. Then the map:
$$\psi\,:\,A\ot B\rightarrow\underset{\leftarrow}{\lim} A_s\ot B_t,\,\,x\mapsto ((\pi_s\ot\rho_t)(x))_{(s,t)}$$
is an isomorphism of unital locally C*-algebras. 
\end{proposition}

\begin{proof}
By definition, $\psi$ is a continuous unital $*$-homomorphism. We will now show that $\psi$ is a bijective and open map. To deduce that $\psi$ is surjective, consider $x=(x_{s,t})_{(s,t)}\in\underset{\leftarrow}{\lim} A_s\ot B_t$, $x_{s,t}=\psi_{(s,t)}(x)$. Consider the set $I:=\{(s,t,p,n): s\in S, t\in T, p\in S(A_s\ot B_t), n\in \mathbb{N})\}$, with the partial ordering given by $(s_0,t_0,p_0,n_0)\leq (s,t,p,n)$ if $s_0\leq s, t_0\leq t, p=p_0\circ(\pi_{s_0,s}\ot \rho_{t_0,t}), n_0\leq n$. Note that the maps $\pi_s\ot\rho_t\,:\, A\odot B\rightarrow A_s\odot B_t$ are surjective and $A_s\odot B_t$ is dense in $A_s\ot B_t$. Thus, for any $(s,t,p,n)\in I$, there exists an element $y_{s,t,p,n}\in A\odot B$ such that $p (x_{s,t}-(\pi_s\ot\rho_t)(y_{s,t,p,n}))<1/n$. We claim that the net $(y_{s,t,p,n})_{(s,t,p,n)\in I}$ is a Cauchy net in $A\ot B$. Indeed, by Remark \ref{Rmkreps}, for any pair of continuous unital $*$-homomorphisms $\pi\,:\, A\rightarrow\Lcal(H)$ and $\rho\,:\,B\rightarrow\Lcal(K)$ there exists $s_0\in S$, $t_0\in T$ and continuous unital $*$-homomorphisms $\widetilde{\pi}\,:\,A_{s_0}\rightarrow\Lcal(H)$, $\widetilde{\rho}\,:\,B_{t_0}\rightarrow \Lcal(K)$ such that $\pi=\widetilde{\pi}\circ\pi_{s_0}$ and $\rho=\widetilde{\rho}\circ\rho_{t_0}$. Fix $\epsilon>0$. Then, for all $(s,t,p, n),(\alpha,\beta,q,m)\geq(s_0,t_0,p_0,n_0)$, where $p\in S(A_{s_0}\ot B_{t_0})$ is given by the continuous homomorphism $\widetilde{\pi}\ot \widetilde{\rho}$ and $n_0>2/\epsilon$, one has:
\begin{eqnarray*}
&&\Vert(\pi\ot\rho)(y_{s,t,p,n}-y_{\alpha,\beta,q,m})\Vert=p_0[(\pi_{s_0}\ot\rho_{t_0})(y_{s,t,p,n}-y_{\alpha,\beta,q,m})]\\
&=&p_0[(\pi_{s_0,s}\ot\rho_{t_0,t})(\pi_s\ot\rho_t)(y_{s,t,p,n})-(\pi_{s_0,\alpha}\ot\rho_{t_0,\beta})(\pi_\alpha\ot\rho_\beta)(y_{\alpha,\beta,q,m})]\\
&\leq&p_0[(\pi_{s_0,s}\ot\rho_{t_0,t})(\pi_s\ot\rho_t)(y_{s,t,p,n})-(\pi_{s_0,s}\ot\rho_{t_0,t})(x_{s,t})]\\
&&+p_0[(\pi_{s_0,s}\ot\rho_{t_0,t})(x_{s,t})-(\pi_{s_0,\alpha}\ot\pi_{t_0,\beta})(x_{\alpha,\beta})]\\
&&+p_0[(\pi_{s_0,\alpha}\ot\rho_{t_0,\beta})(x_{\alpha,\beta})-(\pi_{s_0,\alpha}\ot\rho_{t_0,\beta})(\pi_\alpha\ot\rho_\beta)(y_{\alpha,\beta,q,m})]\\
&=&p_0[(\pi_{s_0,s}\ot\rho_{t_0,t})(y_{s,t,p,n}-x_{s,t})]+p_0[(\pi_{s_0,\alpha}\ot\rho_{t_0,\beta})(y_{\alpha,\beta,q,m}-x_{\alpha,\beta})]\\&<&\epsilon.
\end{eqnarray*}
Let $y:=\lim y_{s,t,p,n}\in A\ot B$. Then, for all $(s,t)\in S\times T$, for any $p\in S(A_s\ot B_t)$ and all $\epsilon>0$, there exists $(s',t',p',n')\in I$ such that $(s',t',p',n')\geq (s,t,p,n)$, with $n>2/\epsilon$ and $p[(\pi_s\ot\rho_t)(y-y_{s',t',p',n})]\leq\epsilon/2$. Hence,
\begin{eqnarray*}
p[\psi_{s,t}\psi(y)-\psi_{s,t}(x)]&=&p[(\pi_s\ot\rho_t)(y)-x_{s,t}]\\
&\leq& p[(\pi_s\ot\rho_t)(y)-(\pi_s\ot\rho_t)(y_{s',t',p',n'})]+p[(\pi_s\ot\rho_t)(y_{s',t',p',n'})-x_{s,t}]\\
&\leq&\epsilon/2+p[(\pi_{s,s'}\ot\rho_{t,t'})(\pi_{s'}\ot\rho_{t'})(y_{s',t',p',n'})-(\pi_{s,s'}\ot\rho_{t,t'})(x_{s',t'})]\\
&\leq&\epsilon/2+p'[(\pi_{s'}\ot\rho_{t'})(y_{s',t',p',n'})-x_{s',t'}\Vert<\epsilon.
\end{eqnarray*}
Since this is true for all $\epsilon>0$, all $(s,t)$ and any $p\in S(A_s\ot B_t)$, we deduce that $\psi(y)=x$. Hence, we conclude that $\psi$ is surjective.

\vspace{0.2cm}

\noindent Next, we show that $\psi$ is injective. Assume that $x\in A\ot B$ is non-zero. Then, there exists continuous unital $*$-homomorphisms $\pi\,:\, A\rightarrow\Lcal(H)$ and $\rho\,:\, B\rightarrow\Lcal(K)$ such that $(\pi\ot\rho)(x)\neq 0$. It follows from Remark \ref{Rmkreps} that there exists $(s,t)\in S\times T$ such that $(\pi_s\ot\rho_t)(x)\neq 0$. Hence, $\psi(x)\neq 0$. 

\vspace{0.2cm}

\noindent Finally, it remains to show that $\psi$ is open. Let $O\subset A\otimes B$ be a non-empty open set. Fix $y_0\in\psi(O)$ and write $y_0=\psi(x_0)$ with $x_0\in O$. Then, there exists continuous unital $*$-homomorphisms $\pi_i\,:\, A\rightarrow\Lcal(H_i)$ and $\rho_i\,:\,B\rightarrow\Lcal(K_i)$, $1\leq i\leq n$ and $\epsilon>0$ such that:
$$\{x\in A\ot B\,:\,\Vert(\pi_i\ot\rho_i)(x-x_0)\Vert\leq\epsilon\text{ for all }1\leq i\leq n\}\subset O.$$
By Remark \ref{Rmkreps}, for all $1\leq i\leq n$, there exists $s_i\in S$, $t_i\in T$, a finite set $F_i=\{p_1^i,p_2^i,...,p_{k_i}^i\}\subset S(A_{s_i}\ot B_{t_i})$ and $\delta_i>0$ such that
$$V:=\{x\in A\ot B\,:\,p_{t}^i[(\pi_{s_i}\ot\rho_{t_i})(x-x_0)]<\delta_i\text{ for all }1\leq i\leq n, p_t^i\in F_i\}$$
satisfies $x\in V\subset O$. Define the open set $W\subset\underset{\leftarrow}{\lim}A_s\ot A_t$ by:
$$W:=\{y=(y_{s,t})_{s,t}\,:\,p_t^i [y_{s_i,t_i}-(y_0)_{s_i,t_i}]<\delta_i\text{ for all }1\leq i\leq n, p_t^i\in F_i\}.$$
By definition of $\psi$ we have $\psi(V)\subseteq W$ and by surjectivity of $\psi$ we see that $W\subseteq \psi(V)$. Hence, $\psi(V)=W$ is open, $y\in\psi(V)\subset O$. Thus, $\psi$ is a continuous, bijective, open map, giving an isomorphism of unital locally C*-algebras between $A\ot B$ and $\underset{\leftarrow}{\lim} A_s\ot A_t$.
\end{proof}

\subsection{Quantum homomorphisms}

\begin{definition}
For $A$ and $B$ be two unital locally C*-algebras we define $\Qhom(A,B)$ as the set of continuous unital $*$-homomorphisms $\rho\,:\,A\rightarrow B\ot C$, where $C=C^*(\mathbb{F}_\infty)/I$, for some closed two sided ideal $I$ of $C^*(\mathbb{F}_\infty)$. For any such $\rho\in \Qhom(A,B)$, we define $A_\rho$ as the unital C*-subalgebra of $C$ generated by $\{(\omega\ot\id)\rho(a)\,:\,a\in A,\,\omega\in B^*\}$. It is easily checked then that $\rho$ maps $A$ into the locally C*-subalgebra $B\otimes A_\rho$ of $B\ot C$.
\begin{itemize}
\item We write $\pi\leq\rho$ if there exists a unital $*$-homomorphism $\psi\,:\,A_\rho\rightarrow A_\pi$ such that
$$(\id\ot\psi)\rho=\pi$$
It is readily checked that the map $\psi$ is surjective in this case.
\item We say that $\pi$ and $\rho$ are \textit{equivalent} and we write $\pi\sim\rho$ if there exists a $*$-isomorphism $\psi\,:\,A_\rho\rightarrow A_\pi$ such that $(\id\ot\psi)\rho=\pi$. It is immediate that $\sim$ defines an equivalence relation on the set $\Qhom(A,B)$. 
\end{itemize}
\end{definition}

\begin{remark}
Let $\pi,\rho\in\Qhom(A,B)$ and suppose that $\pi\leq\rho$. Then, there is a unique unital $*$-homomorphism $\psi\,:\,A_\rho\rightarrow A_\pi$ such that $(\id\ot\psi)\rho=\pi$. We will denote it by $\psi_{\pi,\rho}$. Indeed, if $\psi,\varphi\,:\,A_\rho\rightarrow A_\pi$ are morphisms such that  $(\id\ot\psi)\rho=\pi=(\id\ot\varphi)\rho$. Then, for all $a\in A$ and $\omega\in B^*$ one has $\psi((\omega\ot\id)\rho(a))=\varphi((\omega\ot\id)\rho(a))$. Hence,
$$X:=\{(\omega\ot\id)\rho(a)\,:\,a\in A,\,\,\omega\in  B^*\}\subset E:=\{x\in A_\rho\,:\,\psi(x)=\varphi(x)\}.$$
Since $\psi$ and $\varphi$ are unital $*$-homomorphisms, $E$ is a closed unital $*$-subalgebra of $A_\rho$ and since $A_\rho$ is the C*-algebra generated by $X$, it follows that $E=A_\rho$ so $\psi=\varphi$. Note a similar remark also holds for the equivalence of quantum homomorphisms.
\end{remark}

\begin{proposition}\label{PropOrder}
The following holds.
\begin{enumerate}
\item If $\pi\sim\pi'$ and $\rho\sim\rho'$ then $\pi\leq\rho$ if and only if $\pi'\leq\rho'$.
\item The relation $\leq$ on $\Qhom(A,B)/\sim$ is an order relation which makes $$\qhom(A,B):=\Qhom(A,B)/\sim$$ a directed set.
\end{enumerate}
\end{proposition}

\begin{proof}

\vspace{0.2cm}

\noindent$(1)$ Note that, by exchanging the roles of $\pi$ and $\pi'$ and $\rho$ and $\rho'$, it suffices to show one implication. Let $\psi\,:\,A_\pi\rightarrow A_{\pi'},\varphi\,:\,A_\rho\rightarrow A_{\rho'}$ be isomorphisms such that $(\id\ot\psi)\pi=\pi'$ and $(\id\ot\varphi)\rho=\rho'$. Assume that $\pi\leq\rho$ and let $\nu\,:\,A_\rho\rightarrow A_\pi$ a unital $*$-isomorphism such that $(\id\ot\nu)\rho=\pi$. Define $\nu':=\psi\circ\nu\circ\varphi^{-1}\,:\, A_{\rho'}\rightarrow A_{\pi'}$ and we have $(\id\ot\nu')\rho'=(\id\ot\psi)(\id\ot\nu)\rho=(\id\ot\psi)\pi=\pi'$, finishing the proof.
\vspace{0.2cm}

\noindent$(2).$ $\leq$ is clearly reflexive on $\Qhom(A,B)$ (take $\psi=\id\,:\,A_\pi\rightarrow A_\pi$). Suppose that $\pi\leq\rho$ and $\rho\leq\varphi$ with morphisms $\psi_1\,:\, A_\rho\rightarrow A_\pi$ and $\psi_2\,:\, A_\varphi\rightarrow A_\rho$ such that $(\id\ot\psi_1)\rho=\pi$ and $(\id\ot\psi_2)\varphi=\rho$. Then, $\psi:=\psi_1\circ\psi_2\,:\, A_\varphi\rightarrow A_\pi$ and $(\id\ot\psi)\varphi=(\id\ot\psi_1)(\id\ot\psi_2)\varphi=(\id\ot\psi_1)\rho=\pi$ so that the relation is transitive on $\Qhom(A,B)$. By $(1)$, the relation $\leq$ is well defined on the quotient $\Qhom(A,B)/\sim$ and it is reflexive and transitive on it by the previous part. Suppose now that $\pi\leq\rho$ and $\rho\leq\pi$ so that there exists $\psi\,:\,A_\rho\rightarrow A_\pi$ and $\varphi\,:\, A_\pi\rightarrow A_\rho$ so that $(\id\ot\psi)\rho=\pi$ and $(\id\ot\varphi)\pi=\rho$. Hence, $(\id\ot\psi\varphi)\pi=\pi$ and $(\id\ot\varphi\psi)\rho=\rho$. So, for all $a\in A$ and $\omega\in B^*$ one has:
$$\varphi\psi((\omega\ot\id)(\pi(a)))=(\omega\ot\id)(\pi(a))\quad\text{and}\quad\psi\varphi((\omega\ot\id)(\rho(a)))=(\omega\ot\id)(\rho(a)).$$
Since $A_\pi$ is generated by coefficients of $\pi$ and $A_\rho$ by coefficients of $\rho$ it follows that $\varphi\psi=\id_{A_\pi}$ and $\psi\varphi=\id_{A_\rho}$ so $\pi\sim\rho$ and $\leq$ is antisymmetric on the quotient.

\vspace{0.2cm}

\noindent Let us show that $\qhom(A,B)$ is directed. Let $\rho_1,\rho_2\in\Qhom(A,B)$. Since $A_{\rho_1}\oplus A_{\rho_2}$ is separable, there exists an ideal $I$ in $C^*(\mathbb{F}_\infty)$ and an isomorphism $\psi\,:\,A_{\rho_1}\oplus A_{\rho_2}\rightarrow C^*(\mathbb{F}_\infty)/I$. Consider the unital $*$-homomorphism $\rho_1\oplus\rho_2\,:\,A\rightarrow B\ot (A_{\rho_1}\oplus A_{\rho_2})$ and define:
$$
\rho:=(\id\ot\psi)\circ(\rho_1\oplus\rho_2)\,:\,A\rightarrow B\ot C^*(\mathbb{F}_\infty)/I,$$
Note that $A_\rho\simeq A_{\rho_1}\oplus A_{\rho_2}$ so that we have unital $*$-homomorphism $\psi_k\,:\,A_\rho\rightarrow A_{\rho_k}$ such that $(\id\ot\psi_k)\rho=\rho_k$ for $k=1,2$ i.e. $\rho\geq\rho_1$ and $\rho\geq\rho_2$.
\end{proof}

\subsection{The universal quantum homomorphism}

A unital locally C*-algebra $A$ is called \textit{locally separable} if for all $p\in S(A)$ there exists a countable subset $D\subset A$ such that any $p$-ball in $A$ contains an element of $D$. It is easy to see that it is equivalent to $A_p$ being a separable C*-algebra, for any $p\in S(A)$. Note also that any separable locally C*-algebra is locally separable and a C*-algebra is locally separable as a locally C*-algebra if and only if it is separable. In particular, any projective limit of separable C*-algebras is a locally separable locally C*-algebra.

\begin{theorem}\label{ThmUniversalPro}
Let $A$ and $B$ be any  unital locally C*-algebras. There exists a unique (up to a canonical isomorphism) locally separable unital locally C*-algebra $\U(A,B)$ with a continuous unital $*$-homomorphism $\psi\,:\, A\rightarrow B\ot \U(A,B)$ such that:
\begin{itemize}
    \item If $\U$ is a locally separable  unital locally C*-algebra and $\pi\,:\,A\rightarrow B\ot\U$ is a unital continuous $*$-homomorphism there exists a unique continuous unital $*$-homomorphism $\widetilde{\pi}\,:\,\U(A,B)\rightarrow\U$ such that $(\id\ot\widetilde{\pi})\psi=\pi$.
\end{itemize}
\end{theorem}

\begin{proof}
If follows from Proposition \ref{PropOrder} that, given any complete set of representatives
$$S\subset\Qhom(A,B)$$
for the equivalence relation $\sim$, the relation $\leq$ restricted to $S$ is an order relation. Hence, we may consider the projective system of unital C*-algebras over $S$ given by $A_\pi$, $\pi\in S$ and, for $\pi\leq\rho$, $\psi_{\pi,\rho}\,:\, A_\rho\rightarrow A_\pi$ the unique unital $*$-homomorphism such that $(\id\ot\psi_{\pi,\rho})\rho=\pi$. Define $\U(A,B)$ as the projective limit so that it comes equipped with unital $*$-homomorphisms $\psi_{\pi}\,:\,\U(A,B)\rightarrow A_\pi$. Let $E_\pi\subset A_\pi$ be the unital $*$-algebra generated by $X_\pi:=\{(\omega\ot\id)\pi(a)\,:\,a\in A,\,\,\omega\in B^*\}$. Note that, since all the $A_\pi$ are separable, the locally C*-algebra $\Ucal(A,B)$ is locally separable. Let $\omega\in B^*$ and $a\in A$ and define $x:=((\omega\ot\id)(\rho(a))_{\rho\in S}$. By Proposition \ref{CorTensor} we see that $x\in\Ucal(A,B)$ and by definition $\psi_\pi(x)=(\omega\ot\id)(\pi(a))$. It follows that $X_\pi\subset\psi_\pi(\Ucal(A,B))$ and since $\psi_\pi(\Ucal(A,B))$ is a unital $*$-algebra we find $E_\pi\subset\psi_\pi(\Ucal(A,B))$. We can now deduce from density of $E_\pi$ in $A_\pi$ and Corollary \ref{CorClosed} that $\psi_\pi$ is surjective. By Proposition \ref{PropAMTensor} we may write $B\ot\U(A,B)=\underset{\leftarrow}{\lim}\,B\ot A_\pi$, where the projective system is with the maps $\id\ot\psi_{\pi,\rho}\,:\,B\ot A_\rho\rightarrow B\ot A_\pi$ and with canonical projections given by $\id\ot\psi_\pi\,:\,B\ot\U(A,B)\rightarrow B\ot A_\pi$. Since we have maps $\pi\,:\, A\rightarrow B\ot A_\pi$, for each $\pi\in S$ satisfying $(\id\ot\psi_{\pi,\rho})\rho=\pi$ for all $\pi\leq\rho$, the universal property of the projective limit $B\ot\U(A,B)=\underset{\leftarrow}{\lim}\, B\ot A_\pi$ gives a unital $*$-homomorphism $\psi\,:\,A\rightarrow B\ot\U(A,B)$ such that $(\id\ot\psi_{\pi})\psi=\pi$, for all $\pi\in S$.

\vspace{0.2cm}

\noindent Define $E_\infty\subset \U(A,B)$ to be the unital $*$-algebra generated by $X_\infty:=\{(\omega\ot\id)(\psi(a))\,:\,a\in A,\,\,\omega\in B^*\}$. By Proposition \ref{CorTensor} we have $\psi_{\pi}((\omega\ot\id)\psi(a))=(\omega\ot\id)(\id\ot\psi_{\pi})\psi(a)=(\omega\ot\id)\pi(a)$.
It follows that $X_\pi\subset\psi_{\pi}(E_\infty)$ for all $\pi\in S$. Since $\psi_{\pi}(E_\infty)$ is a unital $*$-algebra, we have $E_\pi\subset\psi_{\pi}(E_\infty)$ for all $\pi\in S$ so $\psi_{\pi}(E_\infty)$ is dense in $A_\pi$ for all $\pi\in S$. We deduce that $E_\infty$ is dense by Lemma \ref{LemmaDensity}.

\vspace{0.2cm}

\noindent We first show the universal property for unital separable C*-algebras for the unital locally C*-algebra $\U(A,B)$ constructed above. Let $\U$ be a unital separable C*-algebras and suppose that $\pi\,:\, A\rightarrow B\ot\U$ is a unital $*$-homomorphism. We may and will assume that $\U=C^*(\mathbb{F}_\infty)/I$. Since $S$ is a complete set of representatives, there exists a unique $\rho\in S$ such that $\rho\sim\pi$. Let $\varphi\,:\,A_\rho\rightarrow A_\pi$ be the unique isomorphism such that $(\id\ot\varphi)\rho=\pi$ and define the continuous unital $*$-homomorphism $\widetilde{\pi}:=\varphi\circ\psi_{\rho}\,:\,\U(A,B)\rightarrow A_\pi$. Then, $(\id\ot\widetilde{\pi})\psi=(\id\ot\varphi)(\id\ot\psi_{\rho})\psi=(\id\ot\varphi)\rho=\pi$. Note that the uniqueness of $\widetilde{\pi}$ is clear, by Proposition \ref{CorTensor}, continuity of $\widetilde{\pi}$ and density of $E_\infty$. This shows the universal property for C*-algebra. 

\vspace{0.2cm}

\noindent We show now that $\Ucal(A,B)$ satisfies the universal property for locally separable unital locally C*-algebras. Let $\U$ be any locally separable unital locally C*-algebra and write $\U=\underset{\leftarrow}{\lim}_{p\in S(\Ucal)}\,\U_p$ its Arens-Michael decomposition with unital $*$-homomorphism $\pi_{p,q}\,:\,\U_q\rightarrow \U_p$ for $p\leq q$ and canonical map $\pi_{p}\,:\,\U\rightarrow\U_p$ and where all the unital C*-algebra $\Ucal_p$ are separable. From the continuous morphism $\pi\,:\, A\rightarrow B\ot\U$ we get continuous morphisms $(\id\ot\pi_{p})\pi\,:\, A\rightarrow B\ot\U_p$ for all $p\in S(\Ucal)$ and we can use the C*-algebra version of the universal property to get continuous morphisms $\nu_p\,:\,\U(A,B)\rightarrow \U_p$ such that $(\id\ot\nu_p)\psi=(\id\ot\pi_{p})\pi$ for all $p\in S(\Ucal)$. Note that, for all $p\leq q$, $\omega\in B^*$ and $a\in A$,
\begin{eqnarray*}
\pi_{p,q}\nu_q((\omega\ot\id)\psi(a))&=&\pi_{p,q}(\omega\ot\id)(\id\ot\nu_q)\psi(a)=\pi_{p,q}\pi_{q}((\omega\ot\id)\pi(a))\\
&=&\pi_{p}((\omega\ot\id)\pi(a))=\nu_p((\omega\ot\id)\psi(a))
\end{eqnarray*}
By density of $E_\infty$ and continuity of the morphisms involved, we deduce that $\pi_{p,q}\nu_q=\nu_p$. Hence, the universal property of the projective limit $\U$ produces a continuous unital $*$-homomophism $\widetilde{\pi}\,:\,\U(A,B)\rightarrow\U$ such that $\pi_{p}\circ\widetilde{\pi}=\nu_p$ hence, $(\id\ot\pi_{p})(\id\ot\widetilde{\pi})\psi=(\id\ot\nu_p)\psi=(\id\ot\pi_{p})\pi$ for all $p\in S(\Ucal)$. Since the maps $\id\ot\pi_{p}$ ($p\in S(\Ucal)$) separates the points on $B\ot\U$ we deduce that $(\id\ot\widetilde{\pi})\psi=\pi$. This proves the universal property for locally separable unital locally C*-algebras and the uniqueness of $\U(A,B)$ now easily follows from this universal property.\end{proof}

\noindent For $A$, $B$ unital locally C*-algebras, we denote by $\rho_{A,B}\,:\,A\rightarrow B\ot\Ucal(A,B)$ the universal quantum homomorphism from $A$ to $B$. Let $\mathcal{C}$ be the category whose objects are locally separable unital locally C*-algebras and morphisms are continuous unital $*$-homomorphisms. Given $B\in{\rm Obj}(\Ccal)$, we consider the functor $F_B\,:\,\Ccal\rightarrow\Ccal$ defined by $F_B(C)=B\ot C$ and $F_B(\varphi)=\id_B\ot \varphi$ for all $C\in{\rm Obj}(\Ccal)$ and all $\varphi\in{\rm Mor}_\Ccal(C,C')$. The following result is a direct consequence of Theorem \ref{ThmUniversalPro}.

\begin{theorem}\label{ThmAdjoint}
The functor $G_B\,:\,\Ccal\rightarrow\Ccal$ defined by $G_B(A)=\Ucal(A,B)$ and, for a morphism $\varphi\in{\rm Mor}_\Ccal(A,A')$, $G_B(\varphi)\,:\,\Ucal(A,B)\rightarrow\Ucal(A',B)$ is the unique continuous unital $*$-homomorphism such that $(\id\ot G_B(\varphi))\rho_{A,B}=\rho_{A',B}\circ\varphi$ is the left adjoint of the functor $F_B$.
\end{theorem}

\begin{proof}
Fix $B\in{\rm Obj}(\Ccal)$. Given $A,C\in{\rm Obj}(\Ccal)$ we consider the map $$\eta_{A,C}\,:\,{\rm Mor}_\Ccal(A,B\ot C)\rightarrow{\rm Mor}_\Ccal(\Ucal(A,B),C)$$
which maps $\varphi\in{\rm Mor}_\Ccal(A,B\ot C)$ to the unique $\eta_{A,C}(\varphi)\in{\rm Mor}_\Ccal(\Ucal(A,B),C)$ such that $$(\id\ot\eta_{A,C})\rho_{A,B}=\varphi.$$
Note that $\eta_{A,C}$ is bijective, its inverse bijection is given by: $$\eta_{A,C}^{-1}\,:\,{\rm Mor}_\Ccal(\Ucal(A,B),C)\rightarrow{\rm Mor}_\Ccal(A,B\ot C), \quad\pi\mapsto(\id\ot\pi)\rho_{A,B}.$$
It is easy to check that $\eta_{A,C}$ is natural in $A$ and $C$.
\end{proof}

\begin{remark}
Theorems \ref{ThmUniversalPro} and \ref{ThmAdjoint} could be deduced from Freyd's adjoint functor Theorem, using only the properties of the abstract category of locally separable unital locally C*-algebra described in Section \ref{SectionLocally}. However, we have chosen to present a self contained and less abstract proof.
\end{remark}

\subsection{The quantum semigroup of the universal quantum endomorphism}

\begin{definition}
A triple $G=(A,\Delta,\varepsilon)$ is called a \textit{unital quantum semigroup} if $A$ is a unital locally C*-algebra and $\Delta\,:\,A\rightarrow A\ot A$, $\varepsilon\,:\,A\rightarrow\C$ are continuous unital $*$-homomorphism such that $(\id\ot\Delta)\Delta=(\Delta\ot\id)\Delta$ and $(\id\ot\varepsilon)\Delta=(\varepsilon\ot\id)\Delta=\id_A$.
\end{definition}

\begin{remark}
It follows from the definition that $A^*$ is a unital semigroup with the convolution product $\omega*\mu:=(\omega\ot\mu)\Delta$ and unit $\varepsilon$. We denote by $A^*_{+,1}\subset A^*$ the set of states on $A$ i.e. the set of $\omega\in A^*$ such that $\omega(x^*x)\geq 0$ for all $x\in A$ and $\omega(1)=1$. Note that the convolution product and $\varepsilon\in A^*_{+,1}$ also defines a structure of unital semigroup on $A^*_{+,1}$.
\end{remark}

\noindent We define the natural quantum semigroup structure on $\U(A):=\U(A, A)$ for a given unital C*-algebra $A$. Let $\rho:A\rightarrow A\otimes \U(A)$ be the canonical continuous unital $\ast$-homomorphism.

\begin{proposition}\label{PropSemiGroup}
The following holds.
\begin{enumerate}
    \item There exists a unique continuous unital $*$-homomorphism $$\Delta\,:\,\Ucal(A)\rightarrow\Ucal(A)\ot\Ucal(A)\text{ such that }(\id\ot\Delta)\rho=(\rho\ot\id)\rho.$$
    \item There exists a unique continuous unital $*$-homomorphism $$\varepsilon\,:\,\U(A)\rightarrow\mathbb{C}\textit{ such that }(\id_A\ot\varepsilon)\rho=\id_A.$$
    \item $(\U(A),\Delta,\varepsilon)$ is a unital quantum semigroup.
\end{enumerate}
\end{proposition}

\begin{proof}
The existence and uniqueness of $\Delta$ and $\varepsilon$ in $(1)$ and $(2)$ is a direct consequence of the locally C*-algebra version of the universal property of $\Ucal(A)$.

\vspace{0.2cm}

\noindent$(3)$. For the co-associativity of $\Delta$, it suffices to show that $(\id\otimes \Delta)\Delta(x)=(\Delta\otimes \id)\Delta(x)$ for $x=(\omega\ot\id)\rho(a)$, where $\omega\in A^*$ and $a\in A$. On the one hand:
\begin{eqnarray*}
(\id\otimes \Delta)\Delta(x)&=&(\id\ot\Delta)\Delta((\omega\ot\id)\rho(a))=(\id\ot\Delta)(\omega\ot\id\ot\id)(\id\ot\Delta)\rho(a)\\
&=&(\id\ot\Delta)(\omega\ot\id\ot\id)(\rho\ot\id)\rho(a)\\
&=&(\omega\ot\id\ot\id\ot\id)(\id\ot\id\ot\Delta)(\rho\ot\id)\rho(a)\\
&=&(\omega\ot\id\ot\id\ot\id)(\rho\ot\id\ot\id)(\id\ot\Delta)\rho(a)\\
&=&(\omega\ot\id\ot\id\ot\id)(\rho\ot\id\ot\id)(\rho\ot\id)\rho(a).
\end{eqnarray*}
On the other hand:
\begin{eqnarray*}
(\Delta\ot\id)\Delta(x)&=&(\Delta\ot\id)(\omega\ot\id\ot\id)(\rho\ot\id)\rho(a)\\
&=&(\omega\ot\id\ot\id\ot\id)(\id\ot\Delta\ot\id)(\rho\ot\id)\rho(a)\\
&=&(\omega\ot\id\ot\id\ot\id)[((\id\ot\Delta)\rho)\ot\id]\rho(a)\\
&=&(\omega\ot\id\ot\id\ot\id)(\rho\ot\id\ot\id)(\rho\ot\id)\rho(a).
\end{eqnarray*}
\noindent It remains to show that $(\id\otimes\varepsilon)\Delta(x)=x=(\varepsilon\otimes\id)\Delta(x)$ for $x=(\omega\ot\id)\rho(a)$:
\begin{eqnarray*}
(\id\otimes\varepsilon)\Delta(x)&=&(\id\ot\varepsilon)\Delta((\omega\ot\id)\rho(a))=(\omega\ot\id\ot\varepsilon)(\id\ot\Delta)\rho(a)\\
&=&(\omega\ot\id)\rho((\id\ot\varepsilon)\rho(a))=(\omega\ot\id)\rho(a)=x
\end{eqnarray*}
We prove in the same way that $(\varepsilon\ot\id)\Delta(x)=x$.
\end{proof}

\noindent For a unital C*-algebra $A$, let $\Lcal(A)$ be the algebra of linear continuous maps from $A$ to $A$ that we view as a unital semigroup, the product being the composition of maps and the unit being $\id_A$. We also denote by $\UCP(A,A)\subset\mathcal{L}(A)$ the subsemigroup of u.c.p. maps from $A$ to $A$.

\begin{proposition}\label{PropStateUCP}
Let $A$ be any unital C*-algebra. The following holds.
\begin{enumerate}
\item The map $\psi\,:\,\Ucal(A)^*\rightarrow\Lcal(A)$, $\omega\mapsto(\id\ot\omega)\rho$ is an homomorphism of unital semigroups.
\item $\psi$ restricts to a unital semigroup homomorphism $\psi\,:\,\U(A)^*_{+,1}\rightarrow{\rm UCP}(A,A)$.
\end{enumerate}
\end{proposition}

\begin{proof}$(1)$. For $\omega,\mu\in\U(A)^*$ we have:
\begin{eqnarray*}
\psi(\omega*\mu)&=&(\id\ot\omega*\mu)\rho=(\id\ot\omega\ot\mu)(\id\ot\Delta)\rho=(\id\ot\omega\ot\mu)(\rho\ot\id)\rho\\
&=&(\id\ot\omega)((\id\ot\id\ot\mu)(\rho\ot\id)\rho(\cdot))=(\id\ot\omega)\rho((\id\ot\mu)\rho(\cdot))=\psi(\omega)\circ\psi(\mu).
\end{eqnarray*}
 Moreover, $\psi(\varepsilon)=(\id\ot\varepsilon)\rho=\id_A$. 

 \vspace{0.2cm}
 
 \noindent$(2)$. It suffices to show that if $\omega$ is a state on $\U(A)$ then, $\psi(\omega)=(\id\ot\omega)\rho\in\Lcal(A)$ is a ucp map. By writing $\omega=\widetilde{\omega}\circ\widetilde{\pi}$ for some quantum homomorphism $\pi\,:\, A\rightarrow A\ot A_\pi$ and some state $\widetilde{\omega}\in A_\pi^*$, we see that $\psi(\omega)=(\id\ot\widetilde{\omega})(\id\ot\widetilde{\pi})\rho=(\id\ot\widetilde{\omega})\pi\in{\rm UCP}(A,A)$.\end{proof}

\noindent We will need the following Lemma.

\begin{lemma}\label{LemInvertible}
A ucp map from $A$ to $A$ is invertible in the unital semigroup $\UCP(A,A)$ if and only if it is an isomorphism of C*-algebras.
\end{lemma}

\begin{proof}
Suppose that the ucp map $\varphi\,:\,A\rightarrow A$ is invertible in $\UCP(A,A)$ so that there is a ucp map $\psi\,:\,A\rightarrow A$ satisfying $\varphi\psi=\psi\varphi=\id_A$. It suffices to show that any element of $A$ is in the multiplicative domain of $\varphi$. Let $a\in A$ and define $b:=\varphi(a)\in A$ since $\varphi$ and $\psi$ are ucp we have $\varphi(a)^*\varphi(a)\leq\varphi(a^*a)$ and $a^*a=\psi(b)^*\psi(b)\leq\psi(b^*b)$. Applying $\varphi$ to the last inequality we get $\varphi(a^*a)\leq b^*b=\varphi(a)^*\varphi(a)$. Hence, $\varphi(a)^*\varphi(a)=\varphi(a^*a)$ for all $a\in A$.
\end{proof}

\noindent Recall that, whenever $A$ is finite dimensional, $\Ucal(A)$ is a unital C*-algebra and Proposition \ref{PropSemiGroup} shows that $(\Ucal(A),\Delta,\varepsilon)$ is a compact quantum semigroup.

\begin{proposition}
Let $A$ be a finite dimensional unital C*-algebra. The following holds.
\begin{enumerate}
    \item The pair $(\Ucal(A),\Delta)$ is a compact quantum group if and only if $A=\C$.
    \item Every invertible element in the unital semigroup $\Ucal(A)^*_{+,1}$ is a character of $\Ucal(A)$.
\end{enumerate}
\end{proposition}

\begin{proof}$(1)$. Consider the unital $*$-homomorphism $\Delta_A:A\rightarrow A\otimes A$ defined by $\Delta_A(a):=1\otimes a$, for all $a\in A$. By the universal property of $\Ucal(A)$, there exists a unique unital $\ast$-homomorphism $\pi:\U(A)\rightarrow A$ such that $(\id\otimes\pi)\rho=\Delta_A$. Note that $\pi$ is clearly surjective. For $x=(\omega\ot\id)\rho(a)\in\U(A)$, where $a\in A$ and $\omega\in A^{\ast}$ one has:
\begin{eqnarray*}
(\pi\otimes\pi)\Delta(x)&=&(\pi\otimes\pi)\Delta((\omega\ot\id)\rho(a))=(\pi\otimes\pi)(\omega\otimes\id\otimes\id)(\id\otimes\Delta)\rho(a)\\
&=&(\pi\otimes\pi)(\omega\otimes\id\otimes\id)(\rho\otimes\id)\rho(a)\\
&=&(\omega\otimes\id\otimes\id)(\id\otimes\pi\otimes\pi)(\rho\otimes\id)\rho(a)=(\omega\otimes\id\otimes\id)((\id\otimes\pi)\rho\otimes\pi)\rho(a)\\
&=&(\omega\otimes\id\otimes\id)(\Delta_A\otimes\pi)\rho(a)= ((\omega\otimes\id)\Delta_A\otimes\id)(\id\otimes\pi)\rho(a)\\
&=&((\omega\otimes\id)\Delta_A\otimes\id)\Delta_A(a)
= \omega(1)(1\otimes a).
\end{eqnarray*}
Moreover,
\begin{eqnarray*}
\Delta_A\pi(x)&=&\Delta_A\pi((\omega\ot\id)\rho(a))=\Delta_A(\omega\otimes\id)(\id\otimes\pi)\rho(a)=\Delta_A((\omega\otimes\id)\Delta_A(a))\\
&=& \omega(1)(1\otimes a).
\end{eqnarray*}
It follows that $(\pi\ot\pi)\Delta=\Delta_A\circ\pi$. If $(\Ucal(A),\Delta)$ is a compact quantum group then $(A,\Delta)$, being a compact quantum sub-semigroup of $(\Ucal(A),\Delta)$, is also a compact quantum group (see \cite[Section 5]{Pa13}). In particular, $\Delta_A(A)(1\ot A)=\C1\ot A$ is dense in $A\ot A$. It implies that $A=\C1$.

\vspace{0.2cm}

\noindent$(2)$. Assume that $\omega$ is invertible in the unital semigroup $\U(A)^*_{+,1}$, then by Proposition \ref{PropStateUCP} and Lemma \ref{LemInvertible}, $\psi(\omega)=(\id\ot\omega)\rho$ is an isomorphism from $A$ to $A$. It follows that $\rho(A)\subseteq\mathcal{D}$, where $\mathcal{D}$ is the multiplicative domain of the ucp map $\id\ot\omega\,:\, A\ot \U(A)\rightarrow A$. Note that $A\ot\C1\subseteq\mathcal{D}$. Hence, $\mathcal{D}$ contains the C*-algebra $C$ generated by $A\ot\C1\cup\rho(A)$. We claim that $C=A\ot \U(A)$. Indeed, assuming that $A=\bigoplus_{\kappa=1}^KM_{N_\kappa}(\C)$ and writing $\rho(a)=\sum_{\kappa,i,j} e^\kappa_{ij}\ot\rho^\kappa_{ij}(a)$ we see that, for all $1\leq\kappa\leq K$ and all $1\leq i,j\leq N_\kappa$ we have $\sum_{t=1}^{N_\kappa}(e^\kappa_{tr}\ot 1)\rho(a)(e^\kappa_{st}\ot 1)=1\ot\rho^\kappa_{rs}(a)\in C$. Since $\U(A)$ is generated by $\{\rho^\kappa_{rs}(a)\,:\,1\leq\kappa\leq K,1\leq i,j\leq N_\kappa, a\in A\}$, we conclude that $1\ot\U(A)\subset C$. Hence, $C=A\ot\U(A)=\mathcal{D}$ and it follows that $\id\ot\omega$ is an homomorphism. Hence, $\omega$ is a character.
\end{proof}

\begin{example}\label{ExMn}
When $A=M_N(\C)$ we write $\Ucal_N:=\Ucal(M_N(\C))$. Let $(e_{ij})_{ij}$ be the canonical matrix units in $M_N(\C)$ and write, for all $a\in M_N(\C)$, $\rho(a)=\sum_{i,j}e_{ij}\ot\rho_{ij}(a)$. A direct computation gives, for all $a\in M_N(\C)$ and all $1\leq i,j\leq N$,
\begin{enumerate}
    \item $\Delta(\rho_{ij}(a))=\sum_{k,l}\rho_{ij}(e_{kl})\ot\rho_{kl}(a)$,
    \item $\varepsilon(\rho_{ij}(a))=a_{ij}$, where $a_{ij}\in \C$ are such that $a=\sum_{ij}a_{ij}e_{ij}$.
\end{enumerate}
Moreover, denoting by $G_N:=U(N)/\mathbb{S}^1 I_N$ the projective unitary group, we have:
\begin{enumerate}\setcounter{enumi}{2}
    \item For $u\in G_N$ there is a unique $\chi_u\in{\rm Sp}(\Ucal_N)$ such that $(\id\ot\chi_u)\rho(a)=uau^*$ $\forall a\in M_N(\C)$.
    \item ${\rm Sp}(\U_N)\subset\U_N^*$ with the convolution product from $\U_N^*$ and unit $\varepsilon$ is a group and the map $\psi\,:\,G_N\rightarrow{\rm Sp}(\U_N)$ $(u\mapsto\chi_u)$ is a group isomorphism.
\end{enumerate}
Let us prove statements $(3)$ and $(4)$. By the universal property of $\U_N$, characters of $\U_N$ corresponds to unital $*$-homomorphisms from $M_N(\C)$ to $M_N(\C)$. Since $M_N(\C)$ is simple and by a dimension argument, those homomorphisms are actually automorphisms of $M_N(\C)$ which are all interior so in bijection with $G_N$. This proves $(3)$ and the bijectivity of the map $\psi$ defined in $(4)$. Note that $\psi^{-1}$ is actually the restriction to $G_n={\rm Aut}(M_N(\C))\subset\Lcal(M_N(\C))$ of the unital semigroup homomorphism defined in Proposition \ref{PropStateUCP}. It follows that $\psi$ is a group isomorphism. It could also be shown by a direct computation, using the formulas in $(1)$ and $(2)$.\end{example}


\newcommand{\etalchar}[1]{$^{#1}$}
\providecommand{\bysame}{\leavevmode\hbox to3em{\hrulefill}\thinspace}
\providecommand{\MR}{\relax\ifhmode\unskip\space\fi MR }
\providecommand{\MRhref}[2]{%
	\href{http://www.ams.org/mathscinet-getitem?mr=#1}{#2}
}
\providecommand{\href}[2]{#2}

\end{document}